%% file: mininumbers.tex
\documentclass[leqno]{birkjour}

\usepackage{amssymb,amsfonts,amsthm,amsgen,amscd}
\usepackage[matrix,arrow,curve]{xy}
\usepackage{graphicx}

\usepackage{cite}
\usepackage{pxfonts}
\usepackage{enumerate}
\usepackage{yhmath}
\usepackage{ulem}
\usepackage{float}
\usepackage{microtype}

\usepackage[pdfborder={0 0 0}]{hyperref}

\numberwithin{figure}{section}
	\newtheorem{thm}{Theorem}[section]
	\newtheorem{cor}[thm]{Corollary}
	\newtheorem{lem}[thm]{Lemma}
	\newtheorem{prop}[thm]{Proposition}
	\newtheorem{dia}[thm]{Diagram}
  \theoremstyle{definition}
	\newtheorem{defn}[thm]{Definition}
	\newtheorem{tble}[thm]{Table}
	\newtheorem{con}[thm]{Conventions and Notations}
\theoremstyle{remark}
	\newtheorem{rem}[thm]{Remark}
	
  \newtheoremstyle{example}{3pt}{3pt}{}{}{\bfseries}{:}{.5em}{}
	\theoremstyle{example}
	\newtheorem{exa}[thm]{Example}
	\newtheorem{war}[thm]{Warning}
	\newtheorem*{example}{Example}

  \renewcommand{\theequation}{\thethm}

	\newcommand{\Z}{\mathbb{Z}}
	\newcommand{\C}{\mathbb{C}}
	\newcommand{\R}{\mathbb{R}}
	\newcommand{\K}{\mathbb{K}}
	\renewcommand{\H}{\mathbb{H}}
	\newcommand{\dis}{\displaystyle}
	\newcommand{\scr}{\scriptstyle}
	\newcommand{\lt}{\left(}
	\newcommand{\rt}{\right)}

	\newcommand{\Bigdownarrow}{\Downarrow}
	
	\newcommand{\scirc}{{\scriptscriptstyle{\circ}}}
	\newcommand{\incl}{\operatorname{incl}}
	\newcommand{\id}{\operatorname{id}}
	\newcommand{\pr}{\operatorname{pr}}
	\newcommand{\stab}{\operatorname{stab}}
	\newcommand{\proj}{\operatorname{proj}}

\begin{document}

\input{Kapitel1}
\input{Kapitel2}
\input{Kapitel3}
\input{Kapitel4}
\input{Kapitel5}
\input{Kapitel6}

\input{Kapitel7}

\newpage

\input{Referenzen}
\end{document}

%% file: Kapitel1.tex
\title[Minimum numbers and Wecken theorems in coincidence theory]{Minimum numbers and Wecken theorems in topological coincidence theory. I}
\author{Ulrich Koschorke}
\address{%
Fachbereich Mathematik\\
Universit\"at Siegen\\
57068 Siegen, Germany}

\email{koschorke@mathematik.uni-siegen.de}

\dedicatory{
Gratefully dedicated to my thesis advisor \\
\large{Richard S. Palais on the occassion of his 80$^{\text{th}}$ birthday}
}

\begin{abstract}
	Minimum numbers measure the obstruction to removing coincidences of two given maps (between smooth manifolds \ $M$ \ and \ $N$, \ of dimensions $m$ and $n$, resp.).
	In this paper we compare them to four distinct types of Nielsen numbers.
	These agree with the classical Nielsen number when \ $m = n$ \ (e.\,g. in the fixed point setting where \ $M = N$ \ and one of the maps is the identity map).
	However, in higher codimensions \ $m - n > 0$ \ their definitions and computations involve distinct aspects of differential topology and homotopy theory.

	We develop tools which help us 1.) to decide when a minimum number is equal to a Nielsen number (``Wecken theorem''), and 2.) to determine Nielsen numbers. Here certain homotopy theoretical criteria play a central r\^ole. E.\,g. failures of the ``Wecken condition'' (cf. definition 1.18 below)
	can have very interesting geometric consequences.
	The selfcoincidence case where the two maps are homotopic turns out to be particularly illuminating.

	We give many concrete applications in special settings where \ $M$ \ or \ $N$ \ are spheres, spherical space forms, projective spaces, tori,
	Stiefel manifolds or Grassmannians.
	Already in the simplest examples an important role is played e.\,g. by Kervaire invariants, all versions of Hopf invariants (\`{a} la James, Hilton, Ganea,\dots), and the elements in the stable homotopy of spheres defined by invariantly framed Lie groups.
\end{abstract}
	\subjclass{Primary 54H25, 55M20; Secondary 55P35, 55Q25, 55Q40}

	\keywords{Coincidence, minimum number, Nielsen number, Wecken condition, loose by small deformation, Kervaire invariant, Hopf invariant, degree}

	\thanks{Supported in part by DFG (Germany) and by the Research in Pairs progamm of MFO (Oberwolfach)}

\maketitle

\section{Introduction and discussion of results} \label{sec:1}

According to Robert Brown (cf. \cite{b}, p.\,9, or \cite{br2}, p. 19), the principal question of topological fixed point theory can be phrased as follows.
\bigskip

\begin{itshape}
	Given a selfmap \ $f$ \ of a topological space \ $M$, what is the minimum number \ $MF(f)$ \ of fixed points among all the maps homotopic to \ $f$?
\end{itshape}
\bigskip

Now, the fixed points of \ $f$ \ are just the points where \ $f$ \ coincides with the identity map id.
It is very natural to consider, more generally, the coincidence set
\stepcounter{thm}
\begin{equation}\label{equ:11}
	C(f_1,f_2) \coloneqq \{ x \in M \vert f_1(x) = f_2(x)\}
\end{equation}
of an arbitrary pair \ $f_1, f_2 \colon M\to N$ \ of (continuous) maps between nonempty smooth connected manifolds without boundary, where \ $M$ \ is compact.

Our principal aim will be to get a thorough understanding of interesting analoga of \ $MF$.
A first candidate is the {\it \textbf{m}inimum number of \textbf{c}oincidence points}
\stepcounter{thm}
\begin{equation}
	MC(f_1,f_2) \;\coloneqq \; \operatorname{min} \{\,\# C(f_1', f_2') \mid f_1' \sim f_1, f_2' \sim f_2 \, \}.
\end{equation}
(If \ $f$ \ is a selfmap of a closed manifold, it follows from a result of R. Brooks \cite{br} that \ $MF(f) = MC (f, \id)$).
However, since in general the dimensions
\stepcounter{thm}
\begin{equation}
	m \coloneqq \dim M \quad, \quad n\coloneqq \dim N
\end{equation}
of the domain and the target may differ, \ $MC (f_1, f_2)$ \ is often infinite and rather crude.
Thus it makes more sense to focus our attention on the (finite!) {\bf m}inimum number of {\bf c}oincidence {\bf c}omponents
\stepcounter{thm}
\begin{equation}
	MCC(f_1, f_2)\;\coloneqq\;\min\{\;\#\pi_0 (C(f_1',f_2')) \,\mid\, f_1'\sim f_1,f_2'\sim f_2\;\}.
\end{equation}
Here \ $\# \pi_0 (C (f_1', f_2'))$ \ denotes the number of {\it pathcomponents} of the coincidence subspace \ $C(f_1', f_2')$ \ of \ $M$ \ (where \ $f_i' \sim f_i$, \ i.\,e. \ $f_i'$ \ is homotopic to \ $f_i$ \ , \ $i=1,2$).

The special case where these minimum numbers vanish is of particular interest (just as in fixed point theory).
\begin{defn}
	We call the pair of maps \ $f_1, f_2 \colon M \to N$ \ {\it loose} if there are homotopies \ $f_1 \sim f_1'$, \ $f_2 \sim f_2'$ \ such that \ $f_1'(x) \not= f_2'(x)$ \ for all \ $x \in M$ \ (i.\,e. \ $f_1, f_2$ \ can be `deformed away' from one another).
\end{defn}
Our approach to studying minimum numbers uses two basic ingredients:
\begin{enumerate}[(i)]
	\item normal bordism theory (stabilized or not); this has also a precise homotopy theoretical description (via the Pontryagin-Thom procedure);
	\item a pathspace \ $E(f_1,f_2)$ \ which could also be named `homotopy coincidence space', in analogy to the terminology of fixed point theory (compare e.\,g. \cite{cj}, II. 6.11).
\end{enumerate}
We obtain three types of looseness obstructions (in section \ref{sec:2} below) and, correspondingly, the Nielsen numbers \ $N^\#(f_1,f_2)$, \
$\tilde N(f_1,f_2)$ \ and \ $N(f_1,f_2)$ \ (see the definitions in section \ref{sec:3}; observe, in particular, the change of notation and warning 3.3 below: ``$\tilde N (f_1, f_2)$'' was denoted by ``$N(f_1, f_2)$'' in previous publications).
These Nielsen numbers are lower bounds for the minimum numbers.
On the other hand, the Reidemeister number \ $\#\pi_0 (E(f_1,f_2))$ \ is an upper bound for \ $MCC(f_1,f_2)$ \ whenever \ $n \not= 2$ \ (cf. \ref{def:1} - \ref{thm:1}\; below).

This suggests very naturally a \textbf{two-step program} for investigating minimum numbers.
First we have to decide when \ $MCC(f_1,f_2)$ \ (or even \ $MC(f_1,f_2)$) is equal to one of the Nielsen numbers and to which one (such results are costumarily called ``Wecken theorems'').
Secondly, we must determine the relevant Nielsen number.
(Here it is helpful that the possible values of Nielsen numbers are often severely restricted).
\bigskip

\begin{example} \textbf{Stiefel and Grassmann manifolds.}
	Let \ $f \colon V_{r,k} \to \stackrel{(\sim)}{G}_{r,k}$ \ be the canonical projection from the Stiefel manifold of orthonormal $k$-frames in \ $\R^r$ \ to the Grassmannian of (nonoriented or oriented) $k$-planes through the origin in \ $\R^r$.

	\begin{thm}\label{thm:6}
		Assume $r \geq 2k \geq 2$. Then
		\begin{equation*}
			MC(f,f) = MCC(f,f) = N^\# (f,f) = \tilde N (f,f) = N(f,f)
		\end{equation*}
		is equal to 0 (or 1, resp.), according as
		\begin{equation*}
			0=2 \chi (G_{r,k}) \, [SO(k)] \;\; \in \; \; \pi_{k (k-1)/2}^S
		\end{equation*}
		(or not, resp.)

		This vanishing condition holds e.\,g. when $k$ is even or $k = 7$ or 9 or $\chi(G_{r,k}) \equiv 0\,(12)$.
	\end{thm}

	Here a fascinating problem enters our discussion: to determine the order of a Lie group, when equipped with a left invariant framing and interpreted -- via the Pontryagin-Thom isomorphism -- as an element in the stable homotopy group of spheres $\pi_{\dis *}^S \cong \Omega_{\dis *}^{fr}$. \
	Deep contributions were made e.\,g. by Atiyah and Smith \cite{as}, Becker and Schulz \cite{bes}, Knapp \cite{kn} and Ossa \cite{o}, to name but a few (consult the summary of results and the references in \cite{o}).
	In particular, it is known that the invariantly framed special orthogonal group \ $SO(k)$ \ is nullbordant for \ $4 \leq k \leq 9$, $k \not= 5$ (cf. table 1 in \cite{o}) and that \ $24[SO(k)] = 0$ \ and \ $2[SO(2l)] = 0$ \ for all $k$ and $l$ (cf. \cite{o}, p. 315, and \cite{bes}, 4.7).

	On the other hand the Euler number $\chi (G_{r,k})$ is easily calculated: it vanishes if $k \not \equiv r \equiv 0(2)$ and equals \ $\begin{pmatrix} [r/2] \\ [k/2] \end{pmatrix}$ otherwise (compare \cite{ms}, 6.3 and 6.4).

	\begin{cor}
		Assume \ $r>k=2$. Then \ $MCC (f,f) =0$.
	\end{cor}
	\begin{cor}
		Assume \ $r \geq k = 3$. Then \ $MCC (f,f) = 0$ if and only if $r$ is even or $r \equiv 1(12)$.
	\end{cor}
	This follows form the fact that \ $[SO(3)] \in \pi_3^S \cong \Z_{24}$ \ has order 12 (cf. \cite{as}).
	\begin{cor}
		Assume \ $r \geq k = 5, \, r \not= 7$. Then \ $MCC (f,f) = 0$ if and only if \ $r \not\equiv 5(6)$.
	\end{cor}
	This follows since \ $[SO(5)]$ \ has order 3 in \ $\pi_{10}^S \cong \Z_6$ \ (cf. \cite{o}).
	\hfill$\Box$
\end{example}

More details of this example can be found in \S3 of \cite{ko2} and in section \ref{sec:5} below.
It is based on the weakest of our three types of Nielsen numbers and on a looseness obstruction in a (stabilized) normal bordism group which -- in this case -- is just the framed bordism group \ $\Omega_{\dis *}^{fr} \cong \pi_{\dis *}^S$.
Of course this leads us deeply into the complicated world of homotopy theory.
However, for large $k$ the easier tools of singular (co)homology theory (with or without twisted coefficients) do not seem to offer the slightest chance to capture any of the interesting coincidence phenomenca described in theorem \ref{thm:6} and its corollaries.

Nevertheless, in order to put these classical methods into perspective, we discuss also a fourth type of a Nielsen number, denoted by \ $N^\Z$, which is based on singular homology with (appropriately twisted) integer coefficients (cf. section \ref{sec:3} below).

When \ $m = n$ \ each of our four types of Nielsen numbers coincides with the classical notion which is so central e.\,g. in topological fixed point theory.

However, in strictly positive codimensions \ $m-n > 0$, \ we get four distinct types of Nielsen numbers
\begin{equation*}
	(\; MC \quad \underset{\not\equiv}{\geq}\quad MCC\quad \underset{\not\equiv}{\geq}\;)\quad N^\# \quad\underset{\not\equiv}{\geq} \quad\tilde N \quad\underset{\not\equiv}{\geq} \quad N \quad\underset{\not\equiv}{\geq} \quad N^\Z \quad\geq \quad0
\end{equation*}
where \ $N^\Z$ \ seems to vanish most of the time (except maybe when e.\,g. aspherical manifolds such as tori are involved).
\vspace{.5ex}

\begin{example}\textbf{maps between spheres.}
	Here our approach allows us to determine all minimum, Nielsen and Reidemeister numbers, thus illustrating the rich variety of possible value combinations.
	In particular, we will see that our four versions of Nielsen numbers yield distinct invariants.
\end{example}
\begin{thm}\label{thm:7}
	Given \ $f_1,f_2 \colon S^m \to S^n$, $m,n \geq 1$, define
	\begin{align*}
		[f] \coloneqq [f_1'] - [a\,\scr\circ \dis f_2'] \in \pi_m (S^n)
	\end{align*}
	where the basepoint preserving maps \ $f_1'$ \ and \ $a\,\scr\circ\dis f_2'$ \ represent the free homotopy classes of \ $f_1$ \ and $a \, \scr \circ \dis f_2$, resp., and a denotes the antipodal involution on \ $S^n$.
	Then
	\begin{subequations}
		\begin{align}
			\# \pi_0 (E(f_1,f_2))=
			\begin{cases}
				1                                & \text{if } n\geq 2; \\
				\vert d^0 (f_1) - d^0 (f_2)\vert & \text{if } m=n=1 \; \text{ and } f_1 \not\sim f_2; \\
				\infty                           & \text{if } n=1 \; \text{ and } f_1 \sim f_2.
			\end{cases}
		\end{align}
		(Here and subsequently \ $d^0$ \ denotes the classical mapping degree).
		\begin{align}
			MC (f_1,f_2) =
			\begin{cases}
				0 & \text{if } f_1 \sim a\, \scr \circ \dis f_2; \\
				1 & \text{if } m,n \geq 2 \,\text{ and } [f] \in E (\pi_{m-1} (S^{n-1}))\backslash\{0\}; \\
				\vert d^0 (f_1) - d^0 (f_2) \vert & \text{if } m = n = 1; \\
				\infty & \text{if }m > n \geq 2 \text{ and } [f] \not \in E (\pi_{m-1}(S^{n-1})).
			\end{cases}
		\end{align}
		(Here and subsequently \ $E$ \ denotes the Freudenthal suspension homomorphism).
		\begin{align}
			MCC (f_1,f_2) = N^\# (f_1,f_2)=
			\begin{cases}
				0 & \text{if } f_1 \sim a \, \scr \circ \dis f_2; \\
				\# \pi_0 (E(f_1,f_2)) & \text{if } f_1 \not \sim a \, \scr \circ \dis f_2.
			\end{cases}
		\end{align}
		If \ $n=1$ \ or \ $f_1 \sim a \, \scr \circ \dis f_2$, then
		\begin{align}
			MC (f_1,f_2) = MCC (f_1,f_2) = N^\# (f_1,f_2) = \tilde N (f_1,f_2) = N (f_1,f_2) = N^\Z (f_1,f_2).
		\end{align}
		Thus assume that \ $n \geq 2$ \ and \ $f_1 \not \sim a \scr \circ \dis f_2$. Then
		\begin{align}
			1 = MCC (f_1,f_2) = N^\# (f_1,f_2) \underset{\not\equiv}{\geq} \tilde N (f_1,f_2)
			\underset{\not\equiv}{\geq} N (f_1,f_2) \underset{\not\equiv}{\geq} N^\Z (f_1,f_2) \geq 0;
		\end{align}
	\end{subequations}
	more precisely, \ $\tilde N (f_1,f_2) = 0$ if and only if the stabilized Hopf-James invariant \ $E^\infty (\gamma_k [f]$) (cf. \cite{j}) in the stable homotopy group \ $\pi_{m-1-k(n-1)}^S$ \ of spheres vanishes for all \ $k \geq 1$; in turn, \ $N(f_1,f_2) = 0$ \ if and only if the iterated Freudenthal suspension \ $E^\infty([f]) \in \pi_{m-n}^S$ \ (which is the first Hopf-James invariant) vanishes; moreover \ $N^\Z (f_1,f_2) = 0$ \ whenever $m > n$.

	E.\,g. given maps $f_1,f_2 \colon S^3 \to S^2$, \ we have: \ $\tilde N(f_1,f_2) \not= N(f_1,f_2)$ \ (or \ $N(f_1,f_2) \not= N^\Z(f_1,f_2)$, resp.,) if the classical Hopf invariant of \ $[f]$ \ is even and nontrivial (or odd, resp.).
	Moreover in each of the (infinitely many) dimension combinations \ $(m,n)$ \ listed in \cite{ko4}, 1.17 \ there exist maps \ $f_1,f_2 \colon S^m \to S^n$ \ such that \ $N^\# (f_1,f_2) \not= \tilde N(f_1,f_2)$.
\end{thm}
In contrast, in the following setting the normal bordism approach gives no extra information.
\bigskip

\begin{example}\textbf{maps between tori \ $\mathbf{T^k = (S^1)^k}$.}
	\begin{thm}\label{thm:8}
		For all maps \ $f_1,f_2 \colon T^m \to T^n$, $m,n \geq 1$,
		\begin{equation*}
			MCC (f_1,f_2) = |\! \det (u_1, \dots ,u_n)|
		\end{equation*}
		is equal to all four Nielsen numbers
		\begin{equation*}
			N^\# (f_1,f_2) \; = \; \tilde N (f_1,f_2) \; = \; N(f_1,f_2) \; = \; N^\Z (f_1,f_2).
		\end{equation*}
		Here \ $\det (u_1, \dots ,u_n)$ \ denotes the determinant of an $n \times n$-matrix with integer entries where the column vectors \ $u_i$, $i = 1, \dots ,n$, \ generate the image of
		\begin{equation*}
			f_{1 \dis *} - f_{2 \dis *} \; \colon \; H_1 (T^m;\Z) \to H_1 (T^n;\Z) = \Z^n.
		\end{equation*}
		Moreover,
		\begin{equation*}
			MC (f_1,f_2) =
			\begin{cases}
				MCC (f_1,f_2) & \text{if } m=n \, \text{ or } \, MCC (f_1,f_2) = 0; \\
				\infty        & \text{otherwise;}
			\end{cases}
		\end{equation*}
		and
		\begin{equation*}
			\# \pi_0 (E (f_1,f_2)) = \# \lt H_1 (T^n;\Z) / (f_{1 \dis *} - f_{2 \dis *}) (H_1 (T^m; \Z)) \rt.
		\end{equation*}
	\end{thm}
\end{example}

It may be interesting to note that in the codimension 0 case (i.\,e. when \ $m=n$) \ coincidence problems (for \ $(f_1,f_2)$) are here equivalent to fixed point problems (for \ $f_1-f_2+ \id$). \hfill$\Box$
\bigskip

Details concerning theorems \ref{thm:7} and \ref{thm:8} are given in section \ref{sec:3} below.

It is often possible to describe Nielsen numbers also in terms of covering spaces (cf. \cite{ko7}, 3.4).
This can be used to prove e.\,g. the following result (cf. section \ref{sec:4} below).

\begin{thm}\label{thm:11}
	Let a finite discrete group \ $G$ \ act smoothly and freely on the sphere \ $S^n$ \ and consider two maps \ $f_1,f_2 \colon S^m \to S^n/G$ \ into the resulting orbit manifold, \ $m,n \geq 1$ .

If \ $MCC (f_1,f_2) \not= N^\# (f_1, f_2)$ \ , then \ $f_1 \sim f_2$ \ (i.\,e. \ $f_1, f_2$ \ are homotopic). \medskip
\end{thm}

\begin{defn}\label{def:4}
	A pair \ $(M,N)$ \ of manifolds (as in \ref{equ:11}) has the {\it (full) Wecken property} \ $MCC \equiv N^\#$ \ (or the {\it selfcoincidence Wecken property \ $MCC \equiv N^\#$}, resp.), if \ $MCC (f_1,f_2) = N^\# (f_1,f_2)$ for all maps \ $f_1,f_2 \colon M \to N$ \ (or for all pairs of homotopic maps \ $f_1 \sim f_2 \colon M \to N$, resp.).
\end{defn}

Analoguous properties can be defined for all combinations of a minimum number \ $MC(C)$ \ with a Nielsen number.
E.\,g. according to theorems \ref{thm:7} and \ref{thm:8} all pairs of spheres have the Wecken property \ $MCC \equiv N^\#$ while all pairs of tori enjoy even the Wecken properties
\begin{equation*}
	MCC \; \equiv \; N^\# \; \equiv \; \tilde N \; \equiv \; N \; \equiv \; N^\Z.
\end{equation*}

In a similar parlance one could even summarize the central results of nearly six decades of topological fixed point theory in one single sentence:
given a closed connected manifold \ $M$, \ it has the Wecken fixed point property \ $MF \, \equiv \, N( \id , -)$ \ if and only if it is {\it not} a surface with strictly negative Euler characteristic \ $\chi (M)$ \ (see \cite{ni}, \cite{we} and \cite{ji1}, \cite{ji2}; compare also 3.10 below).

\begin{cor}\label{cor:1}
	If \ $N$ is a spherical space form \ $S^n/G$ \ (as in \ref{thm:11}) and \ $m,n \geq 1$, then the full and the selfcoincidence Wecken properties \ $MCC \equiv N^\#$ \ are equivalent for the pair \ $(S^m,N)$ .
\end{cor}

The same holds also if \ $N$ \ is a real, complex or quaternionic projective space (inspect table 4.7 \ below where the values of \ $M(C)C (f_1, f_2)$ \ and \ $N^\# (f_1, f_2)$ \ are listed for \textit{all} maps \ $f_1,f_2 \colon S^m \to \K P(n')$ ).

So in general is seems worthwhile to take a closer look at the so-called {\it selfcoincidence setting} where the two maps \ $f_1,f_2 \colon M \to N$ \ are homotopic.
Since minimum, Nielsen and Reidemeister numbers are homotopy invariants, we may assume that \ $f_1 \equiv f_2$; \ i.\,e. we need to consider only pairs of the form \ $(f,f)$.

Clearly \ $C(f,f) = M$ \ is connected by assumption.
Therefore \ $MCC(f,f)$ \ and the Nielsen numbers cannot exceed 1.
Moreover, we will see in section \ref{sec:5} below that \ $\tilde N (f,f) = N (f,f)$ \ for all maps \ $f \colon M \to N$.

As another special feature of the selfcoincidence setting we have the following refined notion of looseness (introduced by Dold and Goncalves, cf. \cite{dg}, p.296; compare also \cite{ko7}, 5.3 for further versions).

\begin{defn}
	The pair \ $(f,f)$ \ is {\it loose by small deformation} if and only if for every metric on \ $N$ \ and every \ $\varepsilon > 0$ \ there exists an $\varepsilon$-approximation \ $f'$ of $f$ \ such that \ $f'(x) \not = f(x)$ \ for all \ $x \in M$.
\end{defn}

A homotopy lifting argument shows that \ $(f,f)$ \ is loose by small deformation precisely if the pulled back tangent bundle \ $f^{\dis *}(TN)$ \ has a nowhere  vanishing section over \ $M$ \ (yielding directions into which to `push the map \ $f$ \ away from itself'), cf. \cite{dg}, 2.13 or \cite{ko7}, 5.3. This holds at least when \ $m<n$ \ or when \ $N$ \ allows a vectorfield without zeroes (i.\,e. when \ $N$ \ is noncompact or the Euler characteristic \ $\chi (N)$ \ vanishes, e.\,g. when \ $n$ \ is odd).

In the special case where \ $M = S^m$ \ and -- without loss of generality -- \ $[f] \in \pi_m(N)$, the required section exists if and only if \ $\partial_N ([f]) = 0$, where \ $\partial_N$ \ denotes the boundary homomorphism in the (horizontal) exact homotopy sequence
\stepcounter{thm}
\begin{equation}\label{equ:13}
	\xymatrix{
		\dots \pi_m (STN)  \ar[r]                                                              &
		\pi_m(N)           \ar[r]^-{\partial_N}                                                &
		\pi_{m-1}(S^{n-1}) \ar[r]^{\incl_{\dis *}} \ar@{-->}[d]^{E\coloneqq \text{suspension}} &
		\pi_{m -1} (STN)   \ar[r]                                                              &
		\dots \\
		&& \pi_m (S^n)
	}
\end{equation}
of the space $STN$ of unit tangent vectors (with respect to any Riemannian metric), fibered over \ $N$. (If \ $m = 1$, \ put \ $\partial_N \equiv 0$).

For every element \ $[f] \in \pi_m(N)$ \ we can interpret \ $\partial_N([f])$ \ as being the `index' of a section in \ $f^{\dis *}(TN)$ \ with only one zero in \ $S^m$.
Hence \ $MC(f,f)$ \ cannot exceed 1 \ -- \ just like \ $MCC(f,f)$ \ and the Nielsen numbers.
Clearly, precise vanishing criteria determine these selfcoincidence invariants completely.

Here is an example.

\begin{prop}\label{prop:1}
	Let \ $N = \K P(n')$ \ be a (real, complex or quaternionic) projective space, and let \ $d = 1, 2$ or $4$, resp., denote the real dimension of the field \ $\K = \R,\, \C$ \ or \ $\H$, \ resp.
	Assume that \ $n \, = \, n' \cdot d \, \not\equiv \, 0 \, (2d)$.

	Then \ $\partial_N \equiv 0$. \ Hence for all maps \ $f$ \ from a sphere to \ $N$ \ the pair \ $(f,f)$ is loose by small deformation.
\end{prop}

The proof and many more details concerning this special choice of \ $N$ \ will be given in example \ref{exa:5} below. \hfill $\Box$
\bigskip

Next we express a sufficient condition for the selfcoincidence Wecken property \ $MCC \equiv N^\# $ \ in the language of algebraic topology, as follows.

\begin{defn}\label{def:3}
	We call the assumption
	\begin{equation*}
		0 \;\; = \;\; \partial_N (\pi_m (N)) \; \cap \; \ker (E \colon \pi_{m-1} (S^{n-1}) \to \pi_m (S^n))
	\end{equation*}
	(compare 1.16 
	) the {\it Wecken condition for \ $(m,N)$.}
\end{defn}

From the discussion in \cite{ko7}, 5.6-5.10, and from section \ref{sec:5} below we obtain

\begin{thm}\label{thm:9}
	Given a smooth connected $n$-manifold \ $N$ \ without boundary, let \ $[f] \in \pi_m(N)$, \ $m,n \geq 1$. \ Then
	\begin{equation*}
		MC (f,f) = MCC (f,f)
	\end{equation*}
	and these minimum numbers as well as the four Nielsen numbers of \ $(f,f)$ \ take only 0 \ and 1 as possible values.

	Furthermore we have the following logical implications:

	\medskip
	{\rm (i)} $\partial_N([f])\in\pi_{m-1}(S^{n-1})$ vanishes;

	$\phantom{i}\!$$\Updownarrow$

	{\rm (ii)} $(f,f)$ is loose by small deformation;

	\medskip
	$\phantom{i}$$\Bigdownarrow$ \quad $\left(\Updownarrow \mbox {e.\,g. if } \pi_1(N) \not = 0\; \mbox{ or } N = \K P(n') \mbox{ where } n' \geq 2 \mbox{ and } \K = \R, \C \mbox{ or } \H\right)$

	\medskip
	{\rm (iii)} $MCC(f,f)=0$; equivalently, $(f,f)$ is loose (by any deformation);

	\medskip
	$\phantom{i}$$\Bigdownarrow$ \quad $\left(\Updownarrow \mbox {e.\,g. if } N=S^n/G,\ G \not \cong \Z_2\right)$

	\medskip
	{\rm (iv)} $N^\#(f,f)=0$;

	$\phantom{i}$$\Updownarrow$

	{\rm (v)} $E \,\scr \circ \dis \partial_N (([f])=0$.
	\medskip

	In particular, the five conditions (i) - (v) are equivalent \ (and then $MC(f,f) = MCC(f,f) = N^\# (f,f))$ \ for all maps \ $f \colon S^m \to N$ \ if and only if the Wecken condition holds for \ $(m,N)$ \ (cf. \ref{def:3}).

	This condition is satisfied e.\,g. when $N$ is noncompact or has zero Euler characteristic \ $\chi (N)$ \ (e.\,g. for odd $n$) \ or in the `stable
	dimension range' \ $m < 2n-2$ \ \ or when \ $m \leq n+4$, \ $(m,n) \not= (10,6)$.
\end{thm}

In view of this theorem we may say that \ $N^\# (f,f)$ \ is `\textit{at most} one desuspension short` of being a complete looseness obstruction
(whenever the domain of \ $f$ \ is a sphere).

Clearly the Wecken condition, as well as conditions (i)-(v) except possibly (iii), remain unaffected when we replace the map \ $f$ \ by any of its liftings into a covering space of \ $N$.

If \ $(S^m,N)$ happens to have the selfcoincidence Wecken property \ $MCC \equiv N^\#$ \ (cf. \ref{def:4}) then the Wecken condition for \ $(m,N)$ \ (cf. \ref{def:3}) is equivalent to each loose pair \ $(f,f)$ \ being already loose by small deformation. This holds e.\,g. when \ $N = S^n$.

On the other hand we have

\begin{cor}
	Let \ $N$ \ be a spherical space form \ $S^n/G$ \ (as in \ref{thm:11}), or else a real, complex or quaternionic projective space \ $\K P(n')$. \
	Assume that \ $\# G$, $n\geq2$, \ or that \ $n' \geq 2$, \ resp. (i.\,e. \ $N$ \ is not a sphere).

	Then \ $(S^m , N)$ \ has the full Wecken property \ $MCC \equiv N^\#$ \ if and only if the Wecken condition holds for \ $(m,N)$ \ (cf. definitions \ref{def:4} and \ref{def:3}).
\end{cor}

When can failures of the Wecken condition occur, and which geometric consequences do they have?
\bigskip

\begin{example}\textbf{spherical space forms.}
	Let \ $N = S^n / G$ \ be the orbit manifold of a free smooth action of a \textbf{nontrivial} finite group \ $G$ \ on \ $S^n$ \ as in \ref{thm:11}.
\end{example}

\begin{cor}\label{cor:2}
	Given \ $[f] \in \pi_m(S^n/G)$ \ and a lifting \ $[\tilde f] \in \pi_m (S^n)$ \ of \ $[f], \; m,n \geq 1$, \ the following conditions are equivalent:
	\begin{enumerate}[(i)]
		\item $\partial_N (f) \not= 0$ \ but \ $E \, \scr \circ \dis \partial_N(f) = 0$ \ ;
		\item $MCC (f,f) \not= N^\# (f,f)$ \ ;
		\item $N^\# (f,f) = 0$ \ but \ $f$ \ is coincidence producing (i.\,e. the pair \ $(f,f')$ \ cannot be loose for \textnormal{any}
			map \ $f' \colon S^m \to N$ \ ; thus \ $MCC (f,f') \not= 0$);
		\item $MCC (f,f) > MCC (\tilde f, \tilde f)$ \ ;
		\item $MC (f,f) > MC (\tilde f, \tilde f)$ \ ;
		\item $(\tilde f, \tilde f)$ \ is loose, but \ $(f,f)$ \ is not loose;
		\item $(\tilde f, \tilde f)$ \ is loose, but not by small deformation.
	\end{enumerate}
\end{cor}

All this cannot occur when \ $G \not\cong \Z_2$ \ (since then \ $\chi (N) \cdot \# G \not= \chi (S^n)$ \ or \ $n \not\equiv 0(2)$) \ or when \ $m \leq n+4$ \ (even in the exceptional case \ $m = n+4 = 10$, \ since \ $\pi_{10} (S^6) = 0$) \ or when \ $m = n +5 \not = 11$ \ (since then \ $\ker E = 0$ \ or \ $n \not \equiv 0 (2)$).

However, consider the case \ $(m,n) = (11,6)$.
According to \cite{to} and \cite{pae} we have (in the sequence 1.16 for \ $\tilde N = S^6$)
\begin{equation*}
	\frac12 H \colon \pi_{11} (S^6) \xrightarrow{\cong} \Z \quad ; \quad \pi_{10} (S^5) \cong \Z_2 \quad ; \quad \pi_{10} (V_{7,2}) = 0
\end{equation*}
where \ $H$ \ denotes the Hopf invariant.
Thus \ $\partial_{S^6}$ \ is onto (since \ $ST (S^n) = V_{n+1,2}$, cf. 1.16), but \ $E$ \ and hence \ $E \, \scr \circ \dis \partial_{S^6}$ \ is trivial.
Therefore
\begin{equation*}
	0 \;\; \not= \;\; \pi_{10} (S^5) \;\; = \;\; \partial_{S^6} (\pi_{11} (S^6)) \; \cap \; \ker (E \colon \pi_{10} (S^5) \to \pi_{11} (S^6))
\end{equation*}
and the Wecken condition fails.

Given any map \ $ f \colon S^{11} \to \R P(6)$ and a lifting \ $\tilde f \colon S^{11} \to S^6$ \ of it, we see that
\begin{equation*}
	N^\# (f,f) \; = \; N^\# (\tilde f, \tilde f) \; = \; MCC (\tilde f, \tilde f) \; = \; 0
\end{equation*}
(cf. \ref{thm:9} and \ref{thm:7}c).
If \ $H (\tilde f) \equiv 0 (4)$ \ then \ $\partial_{S^6} (\tilde f) = 0$ \ and both pairs \ $(f,f) , (\tilde f, \tilde f)$ \ are loose by small deformation.
However, if \ $H (\tilde f) \equiv 2 (4)$ \ then \ $\partial_{\R P(6)} (f) \, = \, \partial_{S^6} (\tilde f) \, \not= \, 0$ \ and \ $f$ \ is even coincidence producing (i.\,e. \ $MCC (f,f') \not= 0$ \ for \textit{every} map \ $f' \colon S^{11} \to \R P(6)$ \ whether homotopic to \ $f$ \ or not; cf. theorem \ref{thm:2} below) and \ $MCC (f,f) \not = N^\# (f,f)$;\ moreover \ $(\tilde f, \tilde f)$ \ is loose, but not by small deformation (for further illustrations of the delicate difference between conditions (ii) and (iii) in theorem \ref{thm:9} see e.\,g. \cite{gr1} or \cite{gr2}; compare also \cite{gw}, example 2.4).--

You can find the precise values of \ $MCC (f_1, f_2)$ \ and \ $N^\# (f_1,f_2)$ \ for \textit{all} pairs of maps from \ $S^m$ \ to arbitrary spherical
space forms in theorem \ref{thm:10} below (and in 1.10).  \hfill$\Box$

\bigskip

We may want to look for further target manifolds \ $N$ \ where conditions (ii) and (iii) of theorem \ref{thm:9} are equivalent (and hence so are the Wecken condition and the selfcoincidence Wecken property \ $MCC \equiv N^\#$).
Thus let \ $j_N \colon S^{n-1} \to N - \{x_0\}$ \ denote a (base point preserving) inclusion map of the boundary sphere of a small $n$-ball in \ $N$ \ around some point \ $x_0$.

\begin{thm}\label{thm:2}
	Given \ $[f] \in \pi_m(N)$, we have the following logical implications (campare \ref{thm:9}):
	\medskip

	{\rm (ii)} \ $(f,f)$ \ is loose by small deformation;
	
	$\phantom{i}$ $\Downarrow$

	{\rm (iii)} \ $(f,f)$ \ is loose (by any deformation);

	$\phantom{i}$ $\Downarrow$

	{\rm (iii')} \ $f$ \ is not coincidence producing (i.\,e. there exists some map \ $f' \colon S^m \to N$ \ such the pair \ $(f, f')$ \ is loose,
		cf. \cite{brs});

	$\phantom{i}$ $\Updownarrow$

	{\rm (iii'')} \ $j_{N*} (\partial_N [f]) \in \pi_{m-1} (N - \{x_0\})$ vanishes.
	\medskip

	The three conditions (ii) - (iii') are equivalent for all maps \ $f \colon S^m \to N$ \ if and only if
	\stepcounter{thm}
	\begin{equation}\label{equ:15}
		0 \;\;\; = \;\;\; \partial_N (\pi_m(N)) \; \cap \; \ker (j_{N*} \colon \pi_{m-1} (S^{n-1}) \to \pi_{m-1} (N - \{x_0\})).
	\end{equation}
	This holds e.\,g. when \ $N$ \ is not simply connected or a (real, complex or quaternionic) projective space \ $\K P(n')$, $n' \geq 2$ \ (since then \ $\ker j_{N*} = 0$).
\end{thm}\medskip

Theorems \ref{thm:9} and \ref{thm:2} together yield a criterion for knowing when the Nielsen number \ $N^\# (f,f)$ \ is \textit{precisely} `one desuspension short' of being a complete looseness obstruction.

\begin{cor}
	Assume condition \ref{equ:15}. Then, given \ $[f] \in \pi_m (N)$, \ the pair \ $(f,f)$ \ is loose if and only if \ $\partial_N ([f]) = 0$.
\end{cor}

\begin{rem}
	The requirement \ref{equ:15} is sufficient but not always necessary for conditions (ii) and (iii) in theorem \ref{thm:9} to be equivalent.

	E.\,g. if \ $N = S^n$ \ and \ $m \leq 2n-3$, \ then the Wecken condition \ref{def:3} holds and (ii), (iii) are equivalent.
	However, \ $j_{N \dis *} \equiv 0$ \ and
	\begin{equation*}
		\partial_N (\pi_m (N)) \; \cap \; \ker j_{N \dis *} \; = \; \partial_N (\pi_m (S^n)) \; \cong \; \chi (S^n) \cdot \pi_{m-n}^S \,;
	\end{equation*}
	thus for even \ $n$ \ condition \ref{equ:15} fails to be satisfied whenever \ $2 \cdot \pi_{m-n}^S \not= 0$, \ e.\,g. when \ $m-n = 3, 7, 10, 11, 13, 15, 18$ or $19$ \ (cf. \cite{to}).

	On the other hand, if \ $m = 11$ \ and \ $N = \R P(6)$, \ then \ref{equ:15} holds but \ref{def:3} does'nt (cf. the discussion following \ref{cor:2}).

	We conclude that conditions \ref{def:3} and \ref{equ:15} are independant.
	This is not surprising since -- unlike condition \ref{equ:15} -- the Wecken condition remains always unaffected when we replace a manifold \ $N$ \ by a covering space \ $\tilde N$. \hfill $\Box$
\end{rem}

\bigskip

Here is another consequence of theorems \ref{thm:9} and \ref{thm:2} (it sharpens corollary 1.10 in \cite{ko6}).

\begin{cor}
	If \ $\pi_1 (N)$ \ has a nontrivial proper subgroup \ $G$ \ then for all \ $m \geq 2$ \ every map \ $f \colon S^m \to N$ \ can be  homotoped away from itself by a small deformation.
\end{cor}

Indeed, \ $G$ \ corresponds to a nontrivial covering space \ $\tilde N$ \ of \ $N$ \ with \ $\pi_1 (\tilde N) \not= 0$.
Thus a lifting \ $\tilde f$ \ of \ $f$ \ is not coincidence producing (to see this, pair \ $\tilde f$ \ up with a different lifting);
therefore \ $\partial_{\tilde N} ([\tilde f]) \, = \, \partial_N ([f]) \, = \, 0$. \hfill $\Box$
\bigskip

Now let us analyse a few simple concrete cases where the Wecken condition actually fails to hold.
Again we consider maps into a spherical space form \ $N = S^n/G$ \ as in theorem \ref{thm:11}.
Already in the first nonstable dimension settings we encounter fascinating interrelations with other, seemingly distant, branches of topology.

\begin{thm}\label{thm:13}
	Let \ $\tilde f \, \colon \, S^{2n-2} \, \to \, S^n$ \ be a lifting of a map \ $f \, \colon \, S^{2n-2} \,\to\, N = S^n/\Z_2$.
	Assume that \ $n$ \ is even, \ $n \not= 2,4,8,128$.

	Then the pair \ $(f,f)$ \ is loose if and only if both \ $N^\# (f,f)$ \ and the Kervaire invariant \ $K ([\tilde f])$ \ vanish.
\end{thm}

Originally M. Kervaire introduced his ($\Z_2$-valued) invariant in order to exhibit a triangulable closed manifold which does not admit any differentiable structure (cf. \cite{k}).
Subsequently M. Kervaire and J. Milnor used it in their classification of exotic spheres (cf. \cite{km}).
Then W. Browder showed that \ $K([\tilde f]) = 0$ \ whenever \ $n$ \ is not a power of 2 (cf. \cite{b1}).
But for \ $n = 16,\ 32$ or 64 \ there exist maps \ $f$ \ and \ $\tilde f$ \ as in corollary \ref{cor:2} such that \ $K([\tilde f]) = 1$ \ but \ $N^\# (f,f) = 0$ \ (and hence the seven equivalent conditions (i), \dots (vii) in \ref{cor:2} are all satisfied).
On the other hand, according to the spectacular recent results of M. Hill, M. Hopkins and D. Ravenel \ $K ([ \tilde f]) \equiv 0$ \ whenever \ $n > 128$ \ (cf. \cite{hhr}).
Only the case \ $n = 128$ \ remains open.
This leads us to the following Wecken theorem.
\bigskip

\begin{cor}\label{cor:3}
	Let \ $G$ \ be any nontrivial finite group acting freely and smoothly on \ $S^n$, \ $n\geq 1$.

	Then \ $MCC (f_1, f_2) = N^\# (f_1,f_2)$ \ for all maps \ $f_1, f_2 \, \colon \, S^{2n-2} \to N = S^n/G$ \ \textbf{except} precisely if \ $n = $ 16, 32 or 64 (or maybe 128).
\end{cor}

In the second nonstable dimension setting there are even infinitely many exceptional combinations of \ $m$ \ and \ $n$.

\begin{thm}\label{thm:14}
	Let \ $\tilde f\,\colon\, S^{2n-1} \to S^n$ \ be a lifting of a map \ $f \,\colon\, S^{2n-1} \to N = S^n/\Z_2$.
	Assume that \ $n \equiv 2(4), \; n \geq 6$.

	Then \ $(f,f)$ \ is loose if and only if \ $N^\# (f,f) = 0$ \ and, in addition, the Hopf invariant \ $H (\tilde f)$ \ of \ $\tilde f$ \ is divisible by 4.
\end{thm}

These two conditions are independant.
Indeed, the Hopf invariant homomorphism
\begin{equation*}
	 H \; \colon \; \pi_{2n-1} (S^n) \; = \; \Z \cdot [\iota_n, \iota_n] \oplus \operatorname{torsion} \; \to \Z
\end{equation*}
maps onto \ $2\Z$ \ and kills torsion while \ $N^\# (f,f)$ \ depends only on the torsion part of \ $[\tilde f]$.

\begin{cor}\label{cor:4}
	Let \ $G$ \ be as in \ref{cor:3}.

	Then \ $MCC (f_1,f_2) \, = \, N^\# (f_1,f_2)$ \ for all maps \ $f_1,f_2 \,\colon\, S^{2n-1} \to N= S^n/G$ \ \textbf{except} precisely when \ $n \equiv 2(4)$ \ and \ $n \geq 6$ \ (and therefore \ $G \cong \Z_2$).
\end{cor}

The results \ref{thm:13} - \ref{cor:4} \ are proved with the help of the EHP-sequence.
They center around the question whether certain Whitehead products \ $[\iota_{n-1}, \iota_{n-1}]$ \ and \ $[\iota_{n-1}, \eta_{n-1}]$ \ in the kernel of \ $E$ \ can be halved, i.\,e. lie in \ $2 \cdot \pi_{\dis *} (S^{n-1})$.

In the next six nonstable dimension settings similar Wecken questions still allow fairly complete answers (cf. \cite{kr}).
These tend to get less and less homogeneous as the degree \ $m - 2n + 2$ \ of nonstability (and hence homotopy theoretical complications) increase.
However, we obtain the following simple result.

\begin{thm}
	Let \ $G$ \ be any finite group acting freely and smoothly on \ $S^n$, $n \geq 1$, $n \not= 4,6$. \ If \ $m = 2n +2$ \ or \ $m = 2n + 3$, \ then \
	$MCC (f_1,f_2) = N^\# (f_1, f_2)$ \ for all maps \ $f_1, f_2 \, \colon \, S^m \to S^n/G$.
\end{thm}

In view of theorems 1.12 and \ref{thm:9} this follows from the EHP-Sequence and the fact that the stable homotopy groups \ $\pi_4^S$ \ and \
$\pi_5^S$ \ vanish. \hfill $\Box$
\bigskip

After having discussed at length whether Nielsen numbers are equal to minimum numbers let us try to actually determine their values.
In spite of its apparent strength (`at most one desuspension short of being a complete looseness obstruction'), the Nielsen number \ $N^\# (f,f)$ \ turns out to vanish in a large number of cases (compare \cite{ko8}, 1.25).

\begin{thm}\label{thm:12}
	Let \ $N$ \ be a connected smooth $n$-dimensional manifold without boundary.

	If \ $N^\# (f,f) \not= 0$ \ for some map \ $f \, \colon \, S^m \to N$, then the following restrictions must all be satisfied:
	\begin{enumerate}[a.)]
		\item $n$ \ is even and  \ $m \geq n \geq 4$, \ or else \ $m=2$ \ and \ $N = S^2$ \ or \ $\R P(2)$; \ and
		\item $\pi_1(N) \cong \Z_2$ \ and \ $N$ \ is not orientable, or else \ $\pi_1(N) = 0$ \ (in other words, the first Stiefel-Whitney class of \ $N$ \ induces a \textnormal{monomorphism} from \ $\pi_1(N)$ \ to \ $\Z_2$); \ and
		\item $E \scr \circ \dis \partial_N \not\equiv 0 \; \colon \; \pi_m(N) \,\to\, \pi_m (S^n)$; \ in particular \ $N$ \ is closed and \ $\chi (N) \not= 0$; \; and
		\item if \ $\pi_1(N) \not= 0$ \ then there is no fixed point free selfmap of \ $N$ \ and the homomorphism \ $i_{\dis *} \, \colon \, \pi_m (N-\{ * \} ) \,\to\, \pi_m(N)$ \ (induced by the inclusion of \ $N$, \ punctured at some point) is not onto.
	\end{enumerate}
\end{thm}
\medskip

On the other hand, if e.\,g. \ $n=4,\ 8,\ 12,\ 14,\ 16$ \ or \ 20, \ then there exist infinitely many homotopy classes \ $[f] \in \pi_{2n-1} (\R P(n))$ \ such that \ $N^\# (f,f) \not= 0$.

\begin{cor}
	Assume that at least one of the four restrictions a.), \dots, d.) in \ref{thm:12} is \textnormal{not} satisfied.

	If the Wecken condition for \ $(m,N)$ \ holds (cf. \ref{def:3}) \ then for all maps \ $f \colon S^m \to N$ \ the pair \ $(f,f)$ \ is loose by small deformation.
\end{cor}

Further looseness results along these lines can be found e.\,g. in \cite{ko8}, 1.22 and 1.23.

\begin{thm}\label{cor:5}
	Let \ $k \in \Z \cup \{ \infty \}$ \ be the number elements in \ $\pi_1 (N)$. \
	Then for each pair of maps \ $f_1,f_2 \,\colon\, S^m \to N$, \ $m \geq 2$, \ the Nielsen numbers \ $N^\# (f_1,f_2)$, $\tilde N (f_1,f_2)$, $N (f_1,f_2)$ and $N^\Z (f_1,f_2)$ \ may assume only the values \ $0$ \ or \ $k$ \ or, if (at least) all the restrictions a.),\dots, d.) in \ref{thm:12} are satisfied, also \ $1$ \ as a third possible value.
\end{thm}

Clearly such results can simplify calculations enormously.
E.\,g. if the target manifold \ $N$ \ fails to satisfy just one of the restrictions a.), \dots, d.) we have to decide only whether a given Nielsen number vanishes or not.
(Since this is less likely for \ $N^\#$ \ than e.\,g. for \ $N^\Z$, \ it is in general harder -- but also more rewarding -- to deal with \ $N^\#$).

The proof of theorem \ref{cor:5} (given in section \ref{sec:7} below) combines results involving not only selfcoincidences but also the so-called root case.
Here pairs of the form \ $(f, *)$ \ are studied where \ $*$ \ stands for a constant map.
If \ $m,n \geq 2$ \ the geometry of coincidence data gives rise to a degree homomorphism
\stepcounter{thm}
\begin{equation}
	\deg^\#  \; \coloneqq \; \omega^\# (-,*) \; \colon \; \pi_m (N) \; \to \; \pi_m (S^n \wedge ( \Omega N )^+ )
\end{equation}
(cf. \ref{lem:1} below) which often yields a homogenous approach to dealing with arbitrary pairs \ $(f_1,f_2)$ \ of maps.
Moreover \ $\deg^\#$ \ turns out to be equivalent to an enriched Hopf-Ganea invariant (cf. \cite{ko6}, 7.2 and (64)).
So it is not surprising that some coincidence results such as finiteness criteria for \ $MC (f_1,f_2)$ \ can be expressed in terms of Hopf-Ganea invariants (cf. e.\,g. \cite{ko6}, 7.4 and 7.6).

In the second part of this paper (to be published later) we will discuss a few recent developments and future possibilities concerning minimum numbers and Wecken theorems in the \textit{setting of fiberwise maps}.
\begin{con}
	Throughout this paper \ ${f_1, f_2, f, .. \,\colon \, M \to N}$ \ denote (continuous) maps between connected smooth (non-empty, Hausdorff) manifolds without boundary, having countable bases and (possibly different) strictly positive dimensions \ $m$ \ and \ $n$. \
	We assume \ $M$ \ to be compact (this garantees e.\,g. that \ $MCC (f_1, f_2)$ \ is always finite).
	\ $\sim$ \ means homotopic (or another equivalence relation when this is understood from the context).
	\ $\#$ \ means cardinality or number (in $\{0, 1, 2, \dots, \infty \}$).
\end{con}

%% file: Kapitel2.tex
\section{Looseness obstructions} \label{sec:2}

	When trying to decide whether a given pair \ ${(f_1,f_2)}$ \ is loose we should take a cereful look at the geometry of generic coincidence data.
	\begin{figure}[H]
		\includegraphics[width=.5\textwidth]{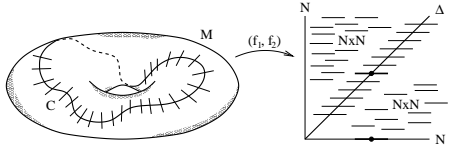}
		\caption{A generic coincidence manifold and its normal bundle}
		\label{fig:1}
	\end{figure}
	After performing an approximation we may assume that the map \ ${(f_1,f_2)\colon M\to N\times N}$ \ is smooth and transverse to the diagonal \newline
	${\Delta=\{(y,y)\in N\times N\mid y\in N\}}$.
	Then the coincidence locus
	\addtocounter{thm}{2}
	\begin{equation}\label{equ:3}
		C=C(f_1,f_2)=(f_1,f_2)^{-1}(\Delta)=\{x\in M\mid f_1(x)=f_2(x)\}
	\end{equation}
	is a closed smooth \ ${(m-n)}$-dimensional submanifold of \ $M$.
	It comes with two important data.
	First there is a commuting diagram of maps
	\stepcounter{thm}
	\begin{equation} \label{equ:12}
		\xymatrix@C=-2.6cm{
			& **[r] E(f_1,f_2) \ar[d]^{\pr} \coloneqq \{ (x,\theta) \in M \times P(N) \mid \theta(0) = f_1(x);\theta(1) = f_2(x)\} \\
			C \ar[ur]^{\tilde g} \ar[r]_{g=\incl} & M
		}
	\end{equation}
	where \ $P(N)$, \ (and pr, resp.), denote the space of all continuous paths \ $\theta \colon [0,1] \to N$, endowed with the compact-open topology, (and the obvious projection, resp.);
	the lifting \ $\tilde g$ \ adds the constant path at \ $f_1(x) = f_2(x)$ \ to \ $g(x) = x \in C$.
	The second datum is the (composite) vector bundle isomorphism
	\stepcounter{thm}
	\begin{equation}\label{equ:10}
		\overline{g}^\# \;\; \colon \;\; \nu (C,M)\; \cong \; ((f_1,f_2) \vert C)^{\dis *}(\nu (\Delta, N \times N)) \; \cong \; f_1^{\dis *} (TN) \vert C
	\end{equation}
	which describes the normal bundle of \ $C$ \ in \ $M$ \ (see figure \ref{fig:1} for an illustration).

	The resulting bordism class
	\stepcounter{thm}
	\begin{equation}\label{equ:2}
		\omega^\# (f_1,f_2) = [C(f_1,f_2), \tilde g, \overline g^\#] \in \Omega^\# (f_1,f_2)
	\end{equation}
	in an appropriate bordism \textit{set} is our strongest -- but also most unmanageable-coincidence invariant (for details and e.\,g. relations to Hopf-Ganea homomorphisms see \cite{ko6}). If the pair \ ${(f_1,f_2)}$ \ is loose, then \ ${\omega^\#(f_1,f_2)}$ \ must necessarily be trivial.
	
	Successive simplifications now yield further looseness obstructions, which we list in the two top lines of diagram \ref{dia:1}.
	Here the maps \textit{stab} and \ ${pr_{\dis{*}} \circ stab}$ \ forget about \ $g\colon C \subset M$ \ being an embedding and stabilize \ $\overline g^\#$, \  i.\,e.

	\begin{equation*}
		\xymatrix @R=0pt @C=22pt {
			\omega^\#(f_1,f_2)                          \ar[r]                 &
			\tilde \omega (f_1,f_2)                     \ar[r]                 &
			\omega(f_1,f_2)                             \ar[r]^-{\mu}          &
			g_{\dis *}([C(f_1,f_2)])                                           \\
			\rotatebox{-90}{$\in$}                                             &
			\rotatebox{-90}{$\in$}                                             &
			\rotatebox{-90}{$\in$}                                             &
			\rotatebox{-90}{$\in$}                                             \\
			\Omega^\# (f_1,f_2)                         \ar[r]^-{\stab}        &
			\Omega_{m-n} (E (f_1,f_2) ; \tilde \varphi) \ar[r]^-{\pr_{\dis *}} &
			\Omega_{m-n} (M;\varphi)                    \ar[r]^-{\mu}      &
			H_{m-n}(M;\tilde \Z_{\varphi})                                     \\
			\infty > N^\#(f_1,f_2)                      \ar@{}[r]|\geq         &
			\tilde N (f_1,f_2)                          \ar@{}[r]|\geq         &
			N (f_1,f_2)                                 \ar@{}[r]|\geq         &
			N^\Z (f_1,f_2) \geq 0
		}
	\end{equation*}

	\begin{dia}\label{dia:1}
		Looseness obstructions (and the corresponding Nielsen numbers, cf. section \ref{sec:3} below)
	\end{dia}
	\bigskip
	
	\noindent
  	  replace it by the induced \textit{stable} tangent bundle isomorphism
  	\begin{equation*}
		\tag{2.4'}
		\overline g \;\;\colon\; TC \oplus \tilde g^{\dis *} \!\lt \pr^{\dis *}\!\!\lt f_1^{\dis *}\lt TN\rt \rt\rt \oplus \uwave{\R^k}  \;\; \cong \; \; \tilde g^{\dis *}\!\lt \pr^{\dis *} (TM) \rt \oplus \uwave{\R^k}
	\end{equation*}
	(or, equivalently,
	\begin{equation*}
		\overline g \; \;\colon\; TC \oplus g^{\dis *} (f_1^{\dis *} (TN)) \oplus \uwave{\R^k} \;\; \cong \;\; g^{\dis *} (TM) \oplus \uwave{\R^k};
	\end{equation*}
	compare \ref{equ:10}), \ $k >> 0$.
	Thus the invariants \smallskip
	\stepcounter{thm}	
	\begin{align}
		\tilde \omega (f_1,f_2) \; &\coloneqq \; [C, \tilde g, \overline g] \; = \; \stab (\omega^\# (f_1,f_2)) & \qquad \text{and} \\
	\stepcounter{thm}\label{equ:14}
		\omega (f_1,f_2) \; &\coloneqq \; [C, g, \overline g] \; = \; \pr_{\dis *}(\tilde\omega (f_1,f_2))
	\end{align}

	\noindent
	lie in the indicated normal bordism groups with coefficients in the virtual vector bundles
	\stepcounter{thm}
	\begin{equation}
		\varphi \coloneqq f_1^{\dis *} (TN) - TM \qquad \text{ and }\qquad \tilde \varphi \coloneqq \pr^{\dis *} (\varphi).
	\end{equation}
	(You may think of normal bordism theory as ``twisted framed bordism''; some background can be found in \cite{da}, \cite{ko1} and also in \cite{sa} where the opposite sign convention is used for coefficient bundles).
	Details about the looseness obstructions \ $\tilde \omega$ \ and \ $\omega$ \ are given in \cite{ko4}.
	Homotopy theoretical versions of the $\tilde \omega$-invariant are discussed in great depth and generality in \cite{cr}; compare also the work of Jaren, Klein and Williams (cf. e.\,g. \cite{kw}).

	Finally note that the Hurewicz homomorphism \ $\mu$ \ in diagram \ref{dia:1} maps \ $\omega (f_1,f_2)$ \ to the image (under the inclusion \ $g$) of the fundamental class of \ $C$ \ with integer coefficients twisted like (the orientation line bundle of) \ $\varphi$.

%% file: Kapitel3.tex
\section{Nielsen and Reidemeister numbers as bounds for minimum numbers} \label{sec:3}

	Each of the coincidence invariants discussed so far seems to have a flaw which makes it either too hard to compute or else too weak: in general, \
	${\Omega^\# (f_1,f_2)}$ \ is only a set without an algebraic structure while \
	$\tilde\omega(f_1,f_2)$ \ lies in a group which, however, may vary with \
	$(f_1,f_2)$; on the other hand, \
	$\mu (\omega (f_1, f_2))$ \ contains no more information than a (co-)homological first order obstruction.

	But we can extract simple numerical invariants which yield lower and upper bounds for our minimum numbers.

	The key is the `pathspace' \
	$E(f_1, f_2)$ \ (sometimes called `homotopy coincidence space' of \
	$(f_1,f_2)$, compare \cite{cj}, II. 6.11), together with the inclusion map \
	$\tilde g$ \ (cf. \ref{thm:6}).

	\begin{defn}
		Two coincidence points \
		$x,x' \in C(f_1, f_2)$ \ are called \textit{Nielsen equivalent} if \
		$\tilde g(x)$ \ and \ $\tilde g(x')$ \ lie in the same pathcomponent of \
		$E(f_1,f_2)$ \ (or, equivalently, if there is a path  \ $c$ \ in \ $M$ \ from \ $x$ \ to \ $x'$ \ such that \
		$f_1 \scr \circ \dis c$ \ and \
		$f_2 \scr \circ \dis c$ \ are homotopic in \ $N$ \ by a homotopy which leaves the endpoints \
		$f_1(x) = f_2(x)$ \ and \ $f_1(x') = f_2(x')$ \ fixed).
	\end{defn}

	This equivalence relation yields a decomposition of the (generic) coincidence space \
	$C = C (f_1, f_2)$ \ into the closed manifolds \
	$C_Q = \tilde g^{-1} (Q), \quad Q\subset E(f_1,f_2)$ \ a pathcomponent.

	\begin{defn}\label{def:2}
		\textit{The Nielsen number} \
		$N^\# (f_1, f_2)$ \ (and \
		$\tilde N(f_1, f_2)$, \ $N(f_1,f_2)$, \ $N^\Z(f_1,f_2)$, resp.) is the number of path components \
		$Q \in \pi_0 (E (f_1, f_2))$ \ such that the coincidence data, when restricted to \
		$C_Q = \tilde g^{-1} (Q)$, contribute nontrivially to \
		$\omega^\# (f_1, f_2)$ \ (and \ $\tilde \omega (f_1, f_2), \; \omega(f_1, f_2), \; \mu (\omega (f_1, f_2)$), resp.).
	\end{defn}

	\begin{war}[change of notation]
		Our Nielsen number \
		$\tilde N(f_1, f_2)$ \ (which may well differ from our Nielsen number \
		$N(f_1, f_2)$ \ (see e.\,g. theorem \ref{thm:7} in the introduction) was previously denoted by \
		$N(f_1, f_2)$ \	(in \cite{ko3} -- \cite{ko11} and \cite{gk}). \hfill$\Box$
	\end{war}\medskip
	
	For an interpretation of Nielsen numbers in terms of covering spaces of \ $N$ \ see \cite{ko7}, \S 3.
	
	Since we assume \ $M$ \ to be compact, \
	$N^\# (f_1, f_2)$, \ $\tilde N(f_1, f_2)$, \ $N(f_1, f_2)$ \ and \
	$N^\Z (f_1, f_2)$ \ must be (\textit{finite}) nonnegative integers.
	Clearly, the weaker an invariant, the less it is able to detect ``essential'' Nielsen equivalence classes.
	This implies the order relations, spelled out in diagram \ref{dia:1}, among our four Nielsen numbers.

	\begin{defn}\label{def:1}
		The \textit{geometric Reidemeister set} and the \textit{Reidemeister number}, resp., of the pair \
		$(f_1,f_2)$ \ is the set \ $\pi_0 (E (f_1, f_2))$ \ of all pathcomponents of \
		$E (f_1, f_2)$ \ (cf. \ref{equ:12}) and its cardinality (
		$\in \{1,2,3,\dots,\infty\}$), resp.
	\end{defn}

	Given any coincidence point \
	$x_0 \in C(f_1, f_2)$ \ (with \
	$y_0 \coloneqq f_1 (x_0) = f_2 (x_0))$, there is a canonical bijection between \
	$\pi_0 (E (f_1, f_2))$ \ and the \textit{algebraic Reidemeister set}
	\addtocounter{thm}{1}
	\begin{equation}\label{equ:1}
		\pi_1(N;y_0) \text{ / Reidemeister equivalence}
	\end{equation}
	\noindent
	where we call \ $[\theta], [\theta'] \in \pi_1 (N, y_0)$ \ \textit{Reidemeister equivalent} if \
	$[\theta'] = f_{1\dis *} (\mu)^{-1} \cdot [\theta] \cdot f_{2 \dis *} (\mu)$ \ for some \
	${\mu \in \pi_1 (M,x_0)}$ \ (compare \cite{ko4}, prop. 2.1).

	\begin{thm}\label{thm:1}
		Let \ $f_1, f_2 \colon M^m \to N^n$ \ be (continuous) maps between connected smooth manifolds (without boundary) of the indicated strictly positive dimensions, \ $M$ \ being compact.
		Then we have
		\begin{enumerate}[(i)]
			\item (Homotopy invariance). The Reidemeister number \ $\#\pi_0 (E (f_1, f_2))$
						\ as well as the Nielsen numbers \ $N^\# (f_1, f_2)$, \
						$\tilde N(f_1, f_2)$, \ $N(f_1, f_2)$ \ and \ $N^\Z(f_1,f_2)$ \ depend
						only on the homotopy classes of \ $f_1$ \ and \ $f_2$.
			\item (Symmetry).
				\begin{align*}
							\# \pi_0 (E (f_1, f_2)) & = \# \pi_0 (E (f_2, f_1)); \\
							N^\# (f_1, f_2)         & = N^\# (f_2, f_1);         \\
							\tilde N (f_1, f_2)     & = \tilde N (f_2, f_1).
						\end{align*}
			\item (Lower bounds for minimum numbers): \ $MCC(f_1,f_2)$ \ is finite and we have
				\begin{equation*}
						MC (f_1, f_2) \geq MCC (f_1, f_2) \geq N^\# (f_1, f_2) \geq \tilde N
						(f_1,f_2) \geq N (f_1, f_2) \geq	N^\Z (f_1, f_2).
				\end{equation*}
			\item (Upper bounds): If \ $n \not= 2$, then
				\begin{equation*}
						MCC (f_1, f_2) \leq \# \pi_0 (E (f_1, f_2));
				\end{equation*}
						if \ $(m,n) \not = (2,2)$ \ then
				\begin{equation*}
					MC (f_1,f_2) \leq \# \pi_0 (E (f_1,f_2)) \text{ \ or \ } MC (f_1,f_2) = \infty.
				\end{equation*}
		\end{enumerate}
	\end{thm}
	\smallskip

	This is proved along the lines of \cite{ko4}, 1.9, and \cite{ko6}, theorem 1.2.
	\bigskip

	In many concrete settings theorem \ref{thm:1} \
	can be used for explicit calculations. In particular, we will be able to prove theorem \ref{thm:7} and the following generalization of theorem \ref{thm:8}.


	\begin{example}\label{exa:1}{\bf maps into tori.}\vspace{-1ex}
		\begin{thm}\label{thm:5}
			Let \ $f_1,f_2\colon M\to T^n\coloneqq(S^1)^n$ \ be any pair of maps from a
			closed, connected, smooth $m$-manifold $M$ into the $n$-dimensional torus
			$\;T^n,\; m,n\geq1$.

			Then its Nielsen, minimum and Reidemeister numbers
			satisfy the inequalities
			\begin{equation*}
				\#\pi_0(E) \;\;\; {\geq} \;\; \mid\det\mid \;\;\; \underset{\overbrace{\text{if }n\not=2}}{\geq} \;\;
        MCC \;\; \geq \;\; N^\# \;\; \geq \;\; \tilde N\;\;\geq\;\;N\;\;\geq\;\; N^\Z\;.
			\end{equation*}
			Here \ $\det$ \ denotes the determinant of an (arbitrary) \ $n\times n$-matrix \
			$(u_1,\dots, u_n)$ \ with integer entries whose column vectors \
			$u_i,\ i=1, \dots, n$, generate the image of
			\begin{equation*}
				f_{1\dis *}-f_{2\dis *} \colon H_1 (M;\Z) \to H_1 (T^n; \Z) = \Z^n.
			\end{equation*}
			Moreover, \ $N^\Z (f_1, f_2) = \vert \det \vert$ \ whenever \
			$(f_1 - f_2)^{\dis *} \colon H^n (T^n; \Z) \to H^n (M; \Z)$ \ is not
			zero; \ if \ $\det \cdot z \not= 0$ \ for all nontrivial \ $z \in H^n(M; \Z)$,
			then \ $N^\Z (f_1, f_2) = 0$ \ whenever \ $(f_1 - f_2)^{\dis *} \equiv 0$ \ on \
			$H^n (T^n; \Z)$.

			\noindent (Here \ $f_{1 \dis *}, f_{2 \dis *}$ \ (and \ $(f_1 - f_2)^{\dis *}$)
			denote the obvious induced homomorphisms in (co)homology).
			
			In particular, if \ $n \not= 2$ \  and \ $(f_1 - f_2)^{\dis *} \not \equiv 0$ \
			on \ $H^n (T^n ; \Z)$, then \ $MCC (f_1, f_2)$ \ is equal to \
			$\vert\!\det (u_1, \dots, u_n) \vert$ \ and to all four Nielsen numbers.

			However, for all \ $m,n \geq 2$ \ and \ $r \in \Z$ \ there exists an $m$-manifold \ $M$ \ and a pair \ $f_1, f_2 \colon M \to T^n$ \ of maps such that
			\begin{equation*}
				MCC \; = \; N^\# \; = \; \tilde N \; = \; N \; = \; N^\Z \; = \; 0 \quad \text{but} \quad
				\det(u_1, \dots ,u_n) \; = \; r.
			\end{equation*}
		\end{thm}

		\begin{proof}
			The torus \ $T^n$ \ has two special features. First, it is an abelian Lie
			Group whose addition induces also an addition of maps into \ $T^n$. We need
			to study only the (`translated') pair ($f \coloneqq f_1 - f_2, \; f_2 - f_2 = 0$)
			since it has the same coincidence behavior as \ $(f_1, f_2)$.

			Secondly, all tori are aspherical. Thus the decomposition
			\begin{equation*}
				f_{\dis *} \; \colon \; H_1 (M; \Z) \; \twoheadrightarrow \; f_{\dis *}
				(H_1 (M; \Z)) \cong \Z^k \; \hookrightarrow \; H_1 (T^n; \Z) = \Z^n
			\end{equation*}
			gives rise to maps
			\begin{equation*}
				M \xrightarrow{f'} T^k \xrightarrow{q} T^n
			\end{equation*}
			whose composite is homotopic to \ $f$ \ (compare \cite{wh}, V, 4.3).

			If \ $\det (u_1,\dots, u_n) = 0$ \ then the generators \ $u_1, \dots, u_n$ \
			of \ $f_{\dis *} (H_1 (M;\Z))$ \ are linearly dependant. Thus \ $k<n$ \ and \ $f$ \
			factors through the lower dimensional torus \ $T^k$ \ (after a suitable
			homotopy). Therefore the pair \ $(f,0)$ \ is loose, its Nielsen and
			minimum numbers vanish, and so does \ $f^{\dis *} (H^n (T^n; \Z))$. The
			Reidemeister number \ $\# \pi_0 (E (f, 0)) = \# (H_1 (T^n; \Z)/ f_{\dis *}
			(H_1(M;\Z))) = \# (\Z^n / q_{\dis *} (\Z^k))$ \	(cf. \ref{def:1} and
			\ref{equ:1}) is infinite here.

			In contrast, if \ $k=n$ then the linearly independant elements \ $u_1, \dots, u_n$ \ (which
			generate the subgroup \ $f_{\dis *} (H_1 (M; \Z))$ \ of \
			$H_1 (T^n; \Z) = \Z^n)$ \ span also a paralleliped in \ $\R^n$ \ whose
			$n$-dimensional volume equals
			\begin{equation*}
				d \coloneqq |\! \det (u_1, \dots, u_n) | \; = \; \# (H_1 (T^n;\Z)/f_{\dis *}
				(H_1(M;\Z)) \;\; = \;\; \# \pi_0 (E (f,0)).
			\end{equation*}
			We may assume that the map \ $q$ \ (in our factorization of \ $f$) is a $d$-fold
			covering map of \ $T^n$ \ with fiber \
			$q^{-1} (\{ 0 \}) = \{ y_1, \dots, y_d \}$. After making \ $f'$ \ transverse to
			the points of this fiber we see that
			\begin{equation*}
				C(f,0) = \coprod_{i=1}^d f'^{-1} (\{ y_i \}).
			\end{equation*}
			This is also the Nielsen decomposition of the coincidence manifold.
			A Nielsen component \ $C_i = f'^{-1} (\{ y_i \}),\; i = 1, \dots, d$, \ makes an
			essential contribution to the (weak, homological) looseness obstruction \
			$\mu (\omega (f, 0))$ \ (cf. diagram \ref{dia:1})
			precisely if its fundamental class \ $g_{\dis *} ([C_i])$ \ in \
			$H_{m-n} (M; \tilde \Z_M)$ \ does not vanish. But this is the Poincar\'{e} dual of \
			$\pm f'^{\dis *} (u)$, where \ $u$ \ generates \ $H^n (T^n; \Z) \cong \Z$
			(compare \cite{ms}, problem 11-C). We conclude that \ $N^\Z (f,0)$ \
			equals 0 or \ $d$, \ according as \ $f'^{\dis *} (u)$ \ vanishes or not. Observe also
			that \ $f^{\dis *} = \det (u_1, \dots u_n) \cdot f'^{\dis *}$ \ on \
			$H^n (T^n; \Z)$.

			In any case the invariants \ $|\!\det| \in [0, \infty)$ \ and
			\ $\# \pi_0 (E) \in (0, \infty]$ \ determine each other, but \ $|\! \det |$
			\ yields the sharper upper bound for \ $MCC$ \ whenever \ $n \not= 2$ \
			(compare with theorem \ref{thm:1},iv).

			Finally, given \ $m,n \geq 2$ \ and \ $r \in \Z$, consider the composed map
			\begin{equation*}
				f \; \colon \; M'(n) \times M'' \; \xrightarrow{\;\proj\;} \; M'(n) \; \xrightarrow{j} \; \bigvee^n
				S^1 \; \xrightarrow{r \vee \id} \; \bigvee^n S^1 \; \xrightarrow{\incl} \; T^n
			\end{equation*}
			where \ $M'(n)$ \ is an oriented surface of genus \ $n$; \ and \ $M''$; \ $\proj$; \ \ 
			$j$; \ \ $r \vee \id$; \ and incl; \ resp., denote an arbitrary closed connected smooth \
			$(m-2)$-manifold; the obvious projection; a map which induces an epimorphism of
			fundamental groups; the wedge of one degree $r$ map with the identity map on
			the remaining wedge of \ $n-1$ \ circles; and the inclusion of the 1-skeleton
			into the \ $n$-torus; resp. Then the pair \ $(f,0)$ \ is loose (since \ $f$ \ is not
			onto), but \ $\det (u_1, \dots, u_n) = r \cdot 1 \cdots 1$.
		\end{proof}
	\end{example}

	\begin{cor} \hspace{-.2cm} (cf. [Ko 4], 1.13).
		For all maps \ $f_1, f_2 \colon M \to S^1$ \ the minimum number \
		$MCC (f_1, f_2)$ \ agrees with the four Nielsen numbers and is characterized by the
		identity
		\begin{equation*}
			(f_{1 \dis *} - f_{2 \dis *}) (H_1 (M; \Z)) \; = \; MCC (f_1, f_2) \cdot H_1 (S^1;\Z).
		\end{equation*}
		In particular, the pair \ $(f_1, f_2)$ \ is loose if and only if \ $f_1$ \ and \
		$f_2$ \ are homotopic. \hfill$\Box$
	\end{cor}

	\noindent
	\textbf{Proof of theorem \ref{thm:8}.}
		We need to apply the arguments of the previous proof only to the case where \ $M = T^m, \ (f_1,f_2) = (f,0)$ \ and \ $k=n$. \
		After suitable `straightening` deformations (cf. \cite{ko11}, \S 2) the maps \ $f$ \ and $f'$ \ are surjective Lie group homomorphisms of tori. Then \ $0 \in T^n$ \ is a regular value of \ $f$ \ and the generic coincidence manifold \ $C(f,0)$ \ consists of affine $(m-n)$-dimensional subtori of \ $T^m$. \
		Each of them is a full (connected!) Nielsen class (and just one point if \ $m=n$); \
		its fundamental class yields a nontrivial element in \ $H_{m-n}(T^m; \Z)$. \ Hence
		\begin{equation*}
		\# \pi_0 (E (f,0)) = \vert \! \det (u_1, \dots, u_n) \vert = MCC (f,0) = N^\Z (f,0).
		\end{equation*}
		In fact, the map \ $\tilde g \colon C (f,0) \to E (f,0)$ \ (cf. \ref{equ:12}) turns out to be even a \textit{homotopy equivalence} (cf. \cite{ko11}, theorem 2.1 (ii)) in this very special setting of tori.
		Therefore, even after applying the \ $mod\ 2$ \ \textit{Hurewicz homomorphism} \ $\mu_2$ \ to \ $\tilde \omega (f,0)$ \ each Nielsen class gets still detected and \ $MCC (f,0)$ \ is also equal to the Nielsen number \ $\tilde N^{\Z_2} (f,0)$ \ based on
		\begin{equation*}
			\mu_2 (\tilde \omega (f,0)) = \tilde g_{\dis *} \lt[C(f,0)]_2\rt \; \in \; H_{m-n} (E (f,0) ; \Z_2)
		\end{equation*}
		(compare \ref{dia:1} and \ref{def:2}).

		If \ $MC (f,0) < \infty$ \ generic coincidence data cannot be detected by higher dimensional homology and hence \ $m = n$ \ or \ $N^\Z (f,0)= 0$.
		\hfill $\Box$
	\bigskip
	
	\noindent
	\textbf{Proof of theorem \ref{thm:7}.}
		The group structure of \ $\pi_m(S^n)$ \ allows us to simplify our arguments. Since \ $f_i$ \ is homotopic to \ $f_i'$ \
		for \ $i = 1,2$ \ and \ $(af_2', f_2')$ \ is loose we see that the pair \
		$([f] = [f_1'] - [a\,\scr \circ \dis f_2'], \;\, 0 = [f_2'] - [f_2'])$ \ has the same minimum, Reidemeister and Nielsen
		numbers as \ $(f_1, f_2)$ \ (compare \cite{ko6}, \S 6). Thus we need to check the claims of theorem \ref{thm:7} only for pairs of
		the form \ $(f,*)$ \ where \ $*$ \ denotes a constant map.

		Then the algebraic description of Reidemeister numbers (cf.\,\ref{equ:1})
		yields \ref{thm:7}a.

		Clearly all minimum and Nielsen numbers vanish if $\;[f]=0$. They agree also
		for a selfmap \ $z \to z^d$, $d \in \Z - \{ 0 \}$, of the unit circle \ $S^1$ \
		in the complex plane (indeed, we may assume that \ $* = 1 \in \C$,
		cf.\,\ref{thm:1}(i), and then each Nielsen class consists of a single
		$d^\text{th}$ root of unity and is obviously neither nullbordant nor
		nullhomologuous; compare e.\,g. \cite{ko4}, 1.13). This proves part of \ref{thm:7}b.

		The remainder of \ref{thm:7}b follows from the finiteness criterion \cite{ko6},
		corollary 6.10.	In spite of the restrictive assumption in theorem \ref{thm:1}(iv)
		above, we need not exclude the case $\; m=n=2$. Indeed, if a	generic map
		\ $f \colon S^2 \to S^2$ \ has a finite set \ $f^{-1}(\{ * \})$ \ of `roots'
		we may find a compact 2-disk \ $D$ \ in the domain and a point \ $y_0$ \ in the
		target such that \ $f^{-1} (\{ * \}) \subset \ring{D}$ \ and \
		$f(D) \subset (S^2 - \{ y_0 \}, *) \cong (\R^2 ,0)$; \ then deform \ $f \mid D$ \
		`linearly' until	\ $*$ \ has a single inverse image point.

		Our statement \ref{thm:7}c is proved in \cite{ko6}, example 1.12.

		The remaining claims in theorem \ref{thm:7}
		follow now from \cite{ko4}, 1.14, 1.15, 1.16, and from the definition of our
		looseness obstructions and Nielsen numbers. \hfill $\Box$

\bigskip
	Finally we turn to the setting of classical fixed point and coincidence theory.

	\begin{exa} {\bf $m = n \geq 1$} (compare Example I in \cite{ko4}).
		Here generic coincidence sets consist of finitely many points, each counted with an
		``index" \ $\pm 1 \in \Z$ \ or	\ $1 \in \Z_2$, according to the orientation
		behavior of \ $\varphi$ \ and	\ $\tilde \varphi$ \ (cf. 2.9). 
		We have:
		\bigskip
		\begin{align*}
			&\omega(f_1,f_2) \; \in \; \Omega_0(M;\varphi)=H_0(M;\tilde \Z_\varphi)=
			\begin{cases}
				\Z & \text{ if } \varphi \text{ is oriented}; \\
				\Z_2 & \text{ if } \varphi \text{ is not orientable.}
			\end{cases} \\
			&\tilde\omega(f_1,f_2) \; \in \; \Omega_0(E(f_1,f_2),\tilde\varphi) =
			\bigoplus_{\substack{Q \in \pi_0 (E (f_1,f_2)) \\
								\text{with } \tilde \varphi \mid Q   \\
								\text{oriented}}}
			\Z \quad \oplus
			\bigoplus_{\substack{Q \in \pi_0 (E (f_1,f_2)) \\
								\text{with } \tilde \varphi \mid Q \\
								\text{non orientable}}}
			\Z_2
		\end{align*}
		(Here the first Stiefel-Whitney classes of \ $M$ \ and \ $N$ \ may help us to
		decide whether a pathcomponent \ $Q$ \ of \ $E (f_1,f_2)$ \ contributes \ $\Z$ \ or \ $\Z_2$ \ as a direct summand to  \
		$\Omega_0 (E (f_1, f_2); \tilde \varphi)$ \ (cf. \cite{ko4}, 5.2): if \ $x_0 \in M$ \ is a coincidence point with \
		$y_0 \coloneqq f_1 (x_0) = f_2 (x_0)$, pick \ $[\theta]$ \ in \ $\pi_1 (N; y_0)$ \ such that \ $(x_0, \theta) \in Q$,
		i.\,e. the element \ $Q \in \pi_0 (E (f_1, f_2))$ \ of the geometric Reidemeister set corresponds -- via the canonical
		bijection, cf. \ref{equ:1} -- to the class of \ $[\theta]$ \ in the algebraic Reidemeister set. Then \
		$\tilde \varphi \mid Q$ \ is orientable if and only if
		\begin{equation*}
			w_1 (M) (\mu) = f_1^{\dis *} (w_1(N)) (\mu)
		\end{equation*}
		for all \ $\mu \in \pi_1 (M, x_0)$ \ such that \
		$f_{2 \dis *} (\mu) = [\theta]^{-1} \cdot f_{1 \dis *} (\mu) \cdot [\theta])$.

		The stabilizing map \
		$\stab \colon \Omega^\# (f_1, f_2) \to \Omega_0 (E (f_1, f_2); \tilde \varphi)$ \
		(cf. \ref{dia:1})
		is bijective except possibly when \ $m = n = 1$.
		\bigskip

		\textit{Special case: fixed point theory.} If \ $f$ \ is a selfmap of \ $M$ \ and we
		consider the coincidences of \ $(f_1, f_2) \coloneqq (\id, f)$ \ (i.\,e. the fixed
		points of \ $f$), then the coefficient bundles \ $\varphi = TM - TM$ \ and \
		$\tilde \varphi$ \ (cf. 2.9)
		are canonically oriented, all indices are integers and according to a theorem of
		Hopf, the index sum \ $\omega (\id, f)$ \ equals the Lefschetz number (cf. \cite{h} and \cite{l}).

		Moreover
		\addtocounter{thm}{1}
		\begin{equation}
			N^\# (\id, f) \; = \; \tilde N(\id, f) \; = \; N(\id, f) \; = \; N^\Z (\id, f)
		\end{equation}
		is the classical Nielsen number \ $N(f)$ \ of \ $f$ \ (cf. \cite{b})

		Since the early 1940's it was known from the work of Nielsen (on surfaces, cf. \cite{ni}) and Wecken (for \ $m \geq 3$,
		cf. \cite{we}) that this Nielsen number agrees with the minimum number \ $MF (f) = MC (\id, f)$ \ of \ $f$ \ whenever \
		$m \not= 2$ \ or the Euler characteristic \ $\chi(M)$ \ is nonnegative. In contrast, B. Jiang proved in 1984/85 that \
		$MF (f) - N (f)$ \ can be strictly positive, cf. \cite{ji1}, \cite{ji2} (and even arbitrarily large,
		cf. \cite{z}, \cite{ke}, \cite{ji3}) for suitable selfmaps of any surface with \
		$\chi(M) < 0$.
	\end{exa}

%% file: Kapitel4.tex
\section{Wecken theorems.} \label{sec:4}

	The inequalities in theorem \ref{thm:1} lead to the following questions.
	\bigskip

	When are they sharp? When is our minimum number $MCC$ (or even $MC$) \textit{equal}
	to a Nielsen number (and, if so, to which one?)
	\bigskip

	Positive results in this direction are costumarily called \textit{Wecken theorems}.
	In our framework they come in different types. The weakest (and most common) type
	would say that $MCC\equiv N^\#$ in a certain setting, while the strongest (but
	most unlikely when $m>n$) type of a Wecken theorem would claim that $MCC$ agrees
	always with $N^\Z$ (and hence also with our other three Nielsen numbers).


	\begin{exa}{\bf maps between spheres}\label{exa:4} (compare theorem \ref{thm:7}).
		Given a fixed dimension combination $(m,n)$ where $\; m,n\geq1$, consider
		arbitrary pairs of maps $\; f_1,f_2\colon S^m\to S^n$.

		If \ $m > n$ \ and \ $\pi_m (S^n) \not= 0$, then
		\begin{equation*}
			MCC\equiv N^\#\not\equiv N^\Z \equiv 0
		\end{equation*}
		and the two intermediate types of a Wecken theorem hold (or not) according as
		the total stabilized Hopf-James homomorphism
		\begin{equation*}
			\Gamma \coloneqq \bigoplus E^\infty \scr \circ \dis
			\gamma_k\colon\pi_m(S^n)\to \bigoplus_{k\geq1} \dis \pi_{m-1-k(n-1)}^S
		\end{equation*}
		(compare \cite{j} or also \cite{ko4}, 1.14-1.17), or the iterated Freudenthal
		suspension homomorphism
		\begin{equation*}
			E^\infty\colon\pi_m(S^n)\to\pi_{m-n}^S
		\end{equation*}
		resp., are injective (or not).

		In the remaining case where \ $m \leq n$ \ or \ $\pi_m (S^n) = 0$ \ we have
		\begin{equation*}
			MC \equiv MCC \equiv N^\# \equiv\tilde N \equiv N \equiv N^\Z.
		\end{equation*}
		This follows e.\,g. from theorem \ref{thm:7} above where the actual values of our
		minimum and Nielsen numbers are also given.		\hfill $\Box$
	\end{exa}
	\bigskip


	\begin{example}{\bf maps from spheres to spherical space forms.}\vspace{-1ex}
	\label{exa:3}
		\begin{thm}\label{thm:10}
			Given a free smooth action of a \textnormal{nontrivial} finite group \ $G$ \ on \ $S^n$, consider maps \
			$f_1,f_2\colon S^m\longrightarrow S^n/G \eqqcolon N$ \ into the resulting orbit manifold, \ $m, n \geq 1$.

			Then
			\begin{equation*}
				MCC (f_1,f_2) = N^\# (f_1,f_2) =
				\begin{cases}
					0 \text{ \ if \ } m<n \text{ \ or \ } (f_1 \sim f_2 \text{ \ and }E(\partial_N(f_1))=0); \\
					1 \text{ \ if \ } f_1 \sim f_2 \text{ \ and \ } E (\partial_N (f_1)) \not= 0; \\
					\#\pi_0 (E (f_1,f_2)) \text{ \ else;}
				\end{cases}
			\end{equation*}
			{\bf except} precisely when \ $f_1,f_2$ \ are homotopic (in the base point free sense) and
			$\partial_N (f_1) \not= 0$ \ but \ $E \scr \circ \dis \partial_N (f_1) = 0$ ; in this case \ $MCC (f_1,f_2) = 1$ \ but \
			$N^\#(f_1,f_2) = 0$. (Such an exception is possible only when \ $n$ \ is even and hence \ $G \cong \Z_2$).

			Moreover the Reidemeister number \ $\# \pi_0 (E(f_1,f_2))$ \ equals the order
			of \ $G$ \ whenever \ $m,n \geq 2$.
		\end{thm}
	\noindent
	\textbf{Proof of theorems \ref{thm:11} and \ref{thm:10}}.
	Clearly vanishing conditions concerning \ $\partial_N (f_1)$ \ or \ $E ( \partial_N (f_1))$ \ are independant of the basepoints chosen when defining
	the boundary homomorphism \ $\partial_N$. \ E.\,g. in the special case where the maps \ $f_1,f_2$ \ are homotopic \ $E( \partial_N (f_1)) = 0$ \ if
	and only if \ $N^\# (f_1,f_2) = 0$ \ (cf. \cite{ko7}, 5.6\,-\,7).
	In particular, the claims in \ref{thm:10} and  \ref{thm:11}, resp., hold for \ $n = 1$ \ and for trivial \ $G$ , resp. (use theorem \ref{thm:7} above).

	Thus we may assume that \ $m,n,\# G \geq 2$ . In view of a result of J. Jezierski (cf. \cite{je}, 4.0) we see that here
	\begin{equation*}
		\# G = \# \pi_0 (E (f_1, f_2)) \geq MCC (f_1, f_2) \geq N^\# (f_1, f_2)
	\end{equation*}
	even when \ $n=2$ \ (compare \ref{equ:1}, \ \ref{thm:1} and \cite{ko7}, 6.2).

	According to \ref{thm:7} and \cite{ko7}, 3.4, we have: \ $\# G \not= N^\# (f_1, f_2)$ \ if and only if there are liftings \
	$\tilde f_1, \tilde f_2 \colon S^m \to S^n$ \ of \ $f_1, f_2$ \ such  that the pair \ $(\tilde f_1, \tilde f_2)$ \ is loose. In turn, this is
	equivalent to \ $f_1, f_2$ \ being homotopic. Indeed, for every nontrivial element \ $g$ \ of \ $G$ \ the pair \
	$(\tilde f_1, g \scr \circ \dis \tilde f_1)$ \ is obviously also loose. Therefore both \ $\tilde f_2$ \ and \ $g \scr \circ \dis \tilde f_1$ \ are
	homotopic to \ $a \scr \circ \dis \tilde f_1$ \ (cf. \cite{gr}, exercise 16.7). Thus \ $g \scr \circ \dis \tilde f_1 \sim \tilde f_2$ \  and \
	$f_1 \sim f_2$. This case is discussed in detail in theorem \ref{thm:9} (to be proved in section 5 below). \hfill $\Box$
	\bigskip

	Here are also a few sample results concerning the minimum number \ $MC$ \ of coincidence \textit{points} (compare e.\,g. \cite{ko6}, 1.13).
	\end{example}
	\begin{prop}\label{prop:2}
		For all maps \ $f_1, f_2 \colon S^m \to N = S^n/G$ \ into an \textnormal{odd-dimensional} spherical space form, \ $n \geq 3$, we have:
		\begin{equation*}
			MC (f_1, f_2) = \begin{cases}
				\infty & \text{if } [f_1] - [f_2] \, \not\in \, p_{\dis *} {\scr \circ} E (\pi); \\
				0      & \text{if } f_1 \sim f_2 \text{ or } m<n; \\
				\# G   & \text{else.}
			\end{cases}
		\end{equation*}
		Here the diagram
		\begin{equation*}
			\pi \; \subset \; \pi_{m-1} (S^{n-1}) \xrightarrow{E} \pi_m (S^n) \xrightarrow{p_{\dis *}} \pi_m (N)
		\end{equation*}
		involves suspension and projection; \ $\pi$ \ denotes all of \ $\pi_{m-1} (S^{n-1})$ \ if \ $\# G \leq 2$ \; and the kernel of the Hopf-Hilton homomorphism
		\begin{equation*}
			h \, \coloneqq \, \bigoplus_{i=0}^{\infty} h_j \, \colon \; \pi_{m-1} (S^{n-1}) \;\to \; \pi_{m-1} (S^{2n-3}) \, \oplus \, \pi_{m-1} (S^{3n-5}) \, \oplus \, \dots
		\end{equation*}
		(cf. \cite{wh}, XI, 8.5) \ if \ $\# G \geq 3$.
	\end{prop}
	Observe that no specific feature of the group action -- apart from the order of \ $G$ \ -- enters the picture here. Such phenomena and the appearance of Hopf invariants here are related to the geometry of (almost) injective points in \cite{ko6}, p.\,655.

	\textit{Special case: \ $m = 4$, $n = 3$}. Here \ $h \colon \pi_3 (S^2) \xrightarrow{\cong} \Z$ \ and \ $E$ \ maps this group onto \ $\pi_4 (S^3) \cong \pi_4 (N) \cong \Z_2$. Hence we have for every map \ $f \colon S^4 \to N$
	\begin{equation*}
		MC (f, *) = \begin{cases}
			\infty & \text{if } [f] \not= 0 \text{ and } \# G \geq 3; \\
			\# G   & \text{if } [f] \not= 0 \text{ and } \# G \leq 2; \\
			0      & \text{if } [f] = 0.
		\end{cases}
	\end{equation*}
	In particular, if \ $\# G \geq 3$ \ and \ $[f] \not= 0$ \ then
	\begin{equation*}
		MC (f, *) = \infty \quad \text{but} \quad MC (\tilde f, \tilde *) = 1
	\end{equation*}
	where \ $\tilde f \colon \; S^4 \to S^3$ \ is a lifting of \ $f$. \hfill $\Box$
	\bigskip


	\begin{exa}{\bf maps from spheres to (real, complex or quaternionic) projective
		spaces.}\label{exa:5}

		Let \ $\mathbb K = \mathbb R$, \ $\mathbb C$ \ or \ $\mathbb H$ \ denote the field of real, complex or quaternionic numbers, and let \ $d = 1, 2,$ or $4$ \ be its real dimension. Let \ $\mathbb K P (n')$ and $V_{n'+1, 2} (\mathbb K)$, \ resp., denote the corresponding space of lines through $0$ and of orthonormal 2--frames, resp., in \ $\mathbb K^{n'+1}$. The real dimension of \ $N = \mathbb K P (n')$ \ is \ $n := d \cdot n'$.

		Consider the diagram
		\addtocounter{thm}{1}
		\begin{equation}\label{equ:16}
		\xymatrix{
			\dots                   \ar[r]                                       &
			\pi_m (V_{n'+1,2} (\K)) \ar[r]^{p_{\K \dis *}}                       &
			\pi_m (S^{n+d-1})       \ar[r]^{\partial_{\K}}  \ar[d]^{p_{\dis *}}  &
			\pi_{m -1} (S^{n-1})    \ar[r]                  \ar[d]^{E}           &
			\dots                                                                \\
			& & \pi_m (\K P (n'))   \ar@{-->}[ur]^{\partial_N}                         &
			\pi_m (S^n)
		}
		\end{equation}
		determined by the canonical fibrations \ $p$ \ and \ $p_{\mathbb K}$; \ $E$ \ denotes the Freudenthal suspension homomorphism.

		In view of theorem \ref{thm:7} above (and the appendix in \cite{ko6}) the following result determines the
		Nielsen and minimum numbers \ $N^\#, \, MCC$ and $MC$ \ for all \ $f_1, f_2 : S^m \to
		\mathbb K P (n'), \ m, n' \ge 1$. \ (Proofs can be found in section 6 of \cite{ko7}).

		\begin{thm}\label{thm:3}
		Assume \ $m, n' \ge 2$. Given \ $[f_i] \in \pi_m (\mathbb K P (n'))$, \ there is a unique homotopy class \ $[\widetilde f_i] \in \pi_m (S^{n + d - 1})$ \ such that \ $p_* ([\widetilde f_i]) - [f_i]$ \ lies in the image of \ $\pi_m (\mathbb K P (n' - 1)), \ i = 1, 2$. (Since this image is isomorphic to \ $\pi_{m -1} (S^{d -1})$, we may assume that \ $\widetilde f_i$ \ is a genuine lifting of \ $f_i$ \ when \ $\mathbb K = \mathbb R$ \ or when \ $m > 2$ \ and \ $\mathbb K = \mathbb C)$. Define \ $[f'_i] := [p \scirc \widetilde f_i ] \in \pi_m (\mathbb K P (n'))$.

		Each pair of homotopy classes \ $[f_1], [f_2] \in \pi_m (\mathbb K P (n'))$ \ satisfies precisely one of the seven conditions which are listed in table \ref{tab:1}, together with the corresponding Nielsen and minimum numbers.

		\end{thm}
		\renewcommand{\arraystretch}{1.3}
		\begin{tabular}{l|c|c|c|} \hline
		{\footnotesize Condition } & $\!\!{\textstyle N^\# (f_1, f_2)}$\!\! & \!\!${\textstyle MCC (f_1, f_2)}$\!\! & \!\!${\textstyle MC (f_1, f_2)}$\!\! \\ \hline\hline
		\footnotesize 1) \ ${\textstyle f'_1 \, \sim \, f'_2,\ \ [\widetilde f_2] \ \in \ \ker \partial_{\mathbb K}}$  & \footnotesize 0 & \footnotesize0 & \footnotesize 0 \\ \hline
		\footnotesize 2) \ ${\textstyle f'_1 \, \sim \, f'_2,\ \ [\widetilde f_2] \ \in \ \ker E \scirc \partial_{\mathbb K} -  \ker \partial_{\mathbb K}}$ & \footnotesize 0 & \footnotesize 1 & \footnotesize 1 \\ \hline
		\footnotesize 3) \ ${\textstyle \mathbb K \, = \, \mathbb R, \ f'_1 \sim f'_2, \ \ \widetilde f_2 \not\sim a \scirc \widetilde f_2}$ & \footnotesize 1 & \footnotesize 1 & \footnotesize 1 \\ \hline
		\footnotesize 4) \ ${\textstyle  \mathbb K  =  \mathbb R,} \ {\textstyle f'_1 \not\sim f'_2, \ \ [\widetilde f_1] - [\widetilde f_2]}  \in  {\textstyle E}\, {\textstyle (\pi_{m -1} (S^{n -1}))} $ \!\!\!\!& \footnotesize 2 & \footnotesize 2 & \footnotesize 2 \\ \hline
		\footnotesize 5) \ ${\textstyle \mathbb K \, = \, \mathbb R,  \ [\widetilde f_1] - [\widetilde f_2] \ \not\in \ E (\pi_{m -1} (S^{n -1}))}$ & \footnotesize 2 & \footnotesize 2 & $\infty$ \\ \hline
		\footnotesize 6) \ ${\textstyle \mathbb K \, = \, \mathbb C \, \text{or} \, \mathbb H, \ [\widetilde f_1] \, = \, [\widetilde f_2] \ \not\in \ \ker E \scirc \partial_{\mathbb K}}$ & \footnotesize 1 & \footnotesize 1 & \footnotesize 1 \\ \hline
		\footnotesize 7) \ ${\textstyle \mathbb K \, = \, \mathbb C \, \text{or} \, \mathbb H, \ [\widetilde f_1] \ne [\widetilde f_2]}$ & \footnotesize 1 & \footnotesize 1 & $\infty$ \\ \hline
		\end{tabular}

		\begin{tble} \label{tab:1}
		Nielsen and minimum coincidence numbers of all pairs of maps \ $f_1, f_2 : S^m \to \mathbb K P (n'), \ m, n' \ge 2$: replace each (possibly base point free) homotopy class \ $[f_i]$ \ by a base point preserving representative and read off the values of \ $N^\#$ \ and \ $M (C)C$. (Here \ $f'_1 \sim f'_2$ \ means that \ $f'_1, f'_2$ \ are homotopic in the basepoint free sense; \ $a$ \ denotes the antipodal map).
		\end{tble}
	\end{exa}
	\medskip

	We conclude also that the triangle in diagram \ref{equ:16} commutes, i.\,e. \ $\partial_{\K} = \partial_N \scr \circ \dis p_{\dis *}$. \ Indeed, compare the exact sequences in \ref{equ:13} and \ref{equ:16} and note that the projection \ $p \,\colon\, S^{n+d-1} \, \to \, N = \K P(n')$ \ pulls the tangent sphere bundle STN over \ $N$ \ back to \ $V_{n'+1,2} (\K)$. \ Moreover the tangent bundle over \ $N$, \ when restricted to the lower dimensional submanifold \ $\K P (n'-1)$, \ allows a nowhere vanishing section. Therefore \ $\partial_N ( \pi_m (\K P(n'))) \, = \, \partial_{\K} (\pi_m (S^{n+d-1}))$ \ and
	\begin{align*}
		[f] \in \ker \partial_N & \Leftrightarrow [\tilde f] \in \ker \partial_{\K}, \\
		[f] \in \ker E \scr \circ \dis \partial_N & \Leftrightarrow [\tilde f] \in \ker E \scr \circ \dis \partial_{\K}
	\end{align*}
	for all \ $[f] \in \pi_m (\K P (n')),$ \ $m \geq 2$. \ Thus we may express table \ref{tab:1} and the Wecken condition \ref{def:3} entirely in terms of \ $\partial_{\K P(n')}$ \ or of \ $\partial_{\K}$, \ as we choose.

	If \ $n'$ \ is odd, then we can multiply the elements of \ $\K^{n'+1}$ \ on the left with the element \ $(0,1)$ \ of the division algebra \ $\K \times \K$ \ (of complex, quaternionic or octonic numbers, resp.). The resulting map
	\begin{align*}
		s \; \colon \; S^{d (n' +1)-1} \; &\longrightarrow \; \K^{n'+1} , \\
		x = (x_1, x_2; x_3, x_4 ; \dots ; x_{n'}, x_{n'+1}) \; &\longrightarrow \; (-\overline{x_2}, \overline{x_1}; - \overline{x_4}, \overline{x_3}; \dots) \not\in \K \cdot x,
	\end{align*}
	yields a section of \ $p_{\K}$. \ Therefore \ $\partial_{\K}$ \ and \ $\partial_{\K P(n')}$ \ vanish identically. Proposition \ref{prop:1} follows. \hfill $\Box$\medskip


	\begin{exa}{\bf maps between tori}
		(compare the proof of theorem \ref{thm:8})

		For all \ $m,n\geq1$ \ and all maps \ $f_1,f_2\colon T^m\to T^n$ \ the minimum
		number \ $MCC(f_1,f_2)$ \ agrees with \ $\vert\! \det(u_1,\dots u_n)\vert$ \ (cf.
		\ref{thm:5}) and all four Nielsen numbers. It is even equal to the Nielsen number
		$\tilde N^{\Z_2}(f_1,f_2)$ which counts those Nielsen classes which contribute
		nontrivially to the \ $mod\ 2$ \ homology class
		\begin{equation*}
			\mu _2(\tilde\omega (f_1,f_2))=\tilde g_{\dis *}([C(f_1,f_2)]_2)\in H_{m-n}(E(f_1,f_2);\Z_2)
		\end{equation*}
		(compare definition \ref{def:2} and diagram \ref{dia:1}).
		\hfill $\Box$
	\end{exa}
	\bigskip

	After all these special examples let us look for results which apply to all
	manifolds $M$ and $N$ (at least in a certain `stable' dimension range). The
	following is a generalisation of Wecken's original result.

	\begin{thm}
		(cf. \cite{ko4}, 1.10).

		Assume $m<2n-2$.

		Then for all maps $f_1,f_2\colon M^m\to N^n$ (as in 1.36)
		we have
		\begin{equation*}
			MCC(f_1,f_2)=N^\#(f_1,f_2)=\tilde N(f_1,f_2).
		\end{equation*}
		\end{thm}
		The proof, given in section 4 of \cite{ko4}, emphasizes the role of the paths
		$\theta$ which occur in \ $E(f_1,f_2)$ \ and in any partial nullbordism of
		\ $(C,\tilde g,\overline{g})$ \ (cf. \ref{equ:12} and 2.7):
		they are a crucial
		ingredient in the construction of homotopies which eliminate inessential
		Nielsen coincidence classes.

	\begin{cor}
		(cf. also \cite{ko6}, 3.3).

		If \ $m=n\not=2$, then
		\begin{equation*}
			MC(f_1,f_2)=MCC(f_1,f_2)=N^\# (f_1,f_2)=\tilde N(f_1,f_2).
		\end{equation*}
		Both minimum numbers agree also with \ $N(f_1,f_2)$ and $N^\Z (f_1,f_2)$ \ under
		the additional assumption that \ $M$ \ and \ $N$ \ are orientable or, more generally, that the vector bundles \ $TM$ \ and \ $f_1^{\dis *}(TN)$ \ over $M$ have the same orientation behavior (i.\,e. \ $w_1(M) \, = \, f_1^{\dis *}(w_1(N))$, \ compare 2.9 and 3.9).
	\end{cor}

%% file: Kapitel5.tex
\section{Selfcoincidences} \label{sec:5}

In this section we discuss our coincidence invariants in the special case where \ $f_1 = f_2 \eqqcolon f \, \colon \, M \to N$. Here we are faced with the unusual situation that the fiber map \ $\pr$ \ in diagram \ref{equ:12} allows the global section \ $s \, \colon \, M \to E (f,f)$ \ defined by
\begin{equation*}
	s(x) \; \coloneqq \; (x, \text{ constant map in } N \text{ at } f(x) ), \quad x \in M.
\end{equation*}
It lifts the inclusion \ $g \, \colon \, C (f_1,f_2) \, \subset \, M$ to yield the coincidence datum \ $\tilde g = s \scr \circ \dis g$ \ (compare \ref{equ:12}). Therefore this lifting contains none of the rich extra information it can capture in the general case when \ $f_1 \not\sim f_2$. \ E.\,g. \ $\omega (f,f)$ \ is just as strong as \ $\tilde \omega (f,f)$ \ (cf. \ref{dia:1}) \ and \ $N (f,f) = \tilde N (f,f)$.
Moreover only the pathcomponent \ $Q_0 \in \pi_0 (E (f,f))$ \ which contains \ $s(M)$ \ can possibly contribute nontrivially to the $\omega$-invariants (cf. definition \ref{def:2}). Thus \ $MCC (f,f)$ \ and the Nielsen numbers of \ $(f,f)$ \ can assume only the values 0 and 1.
Actually the same holds for \ $MC (f,f)$ \ when \ $M = S^m$ \ or in the setting of theorem \ref{thm:6}.

\begin{prop}\label{prop:4}
	Let \ $f \, \colon \, V_{r,k} \, \to \, \overset{(\sim)}{G}_{r,k}$ be the canonical projection (as in \ref{thm:6}) and assume also \ $r \geq 2k \geq 2$.

	If \ $N (f,f) \, = \, 0$ \ then \ $(f,f)$ \ is loose by small deformation.
\end{prop}
\bigskip

\noindent
\textbf{Proof of \ref{thm:6} and \ref{prop:4}.}
	The Grassmannian \ $\tilde G_{r,k}$ \ of oriented $k$-planes in \ $\R^r$ \ allows a tangent vector field \ $v$ \ having a single zero with index \ $\chi (\tilde G_{r,k}) = 2 \chi (G_{r,k})$ \ at the point \ $y_0 = [\R^k \subset \R^r]$ \ in \ $\tilde G_{r,k}$. \
	The corresponding section \ $v \scr \circ \dis f$ \ of the pullback \ $f^{\dis *} (T \tilde G_{r,k})$ \ vanishes only in the fiber \ $f^{-1} ( \{y_0 \}) \, = \, SO(k)$. \
	Since \ $r \geq 2k$ \ by assumption, we can rotate the vectors \ $v_1, \dots, v_k$ \ of a $k$-frame in \ $\R^k$ \ to (points near) the standard basis vectors \ $e_{k+1}, \dots, e_{2k}$ \ in \ $\R^r$. \
	This yields an isotopy which deforms the inclusion map \ $g \, \colon \, SO(k) \, \subset \, V_{r,k}$ \ (compare \ref{equ:12}) into a small neighbourhood of the point \ $y_0 \, = \, (e_{k+1}, \dots , e_{2k})$ \ in \ $V_{r,k}$. \
	Conversely there is a ball \ $B$ \ in \ $V_{r,k}$ \ which contains the zero set \ $SO(k)$ \ of \ $v \scr \circ \dis f$. \
	Actually after a suitable modification in \ $B$ \ this section vanishes in at most one point.
	Thus \ $MC (f,f) \leq 1$. \
	Moreover, if we choose a trivialization of the restricted vector bundle \ $f^{\dis *} (T\tilde G_{r,k}) \vert B$ \ we obtain the `index`
	\begin{equation*}
		i (v \scr \circ \dis f, B) \; \coloneqq \; \left[ \lt v \scr \circ \dis f \vert \partial B\rt / \| v \scr \circ \dis f \vert \partial B \| \right] \; \in \; \left[ \partial B, S^{n-1} \right] \; \approx \; \left[ S^{m-1}, S^{n-1} \right],
	\end{equation*}
	i.\,e. the local obstruction to removing these singularities.

	Let us compare this index with the $\omega$-invariants of \ $(f,f)$. \
	A fixed choice of an oriention of \ $SO(k)$ \ equips this Lie group with a left invariant framing; it also yields a (stable) trivialization of the tangent bundle along the fibers of \ $f$, \ and hence of the coefficient bundle  \ $\varphi \, = \, f^{\dis *} (TN) - TM$ \ (cf. 2.9 and diagram \ref{dia:1}).
	Thus \ $\omega (f,f)$ \ lies in the \textit{framed} bordism group
	\begin{equation*}
		\Omega_{m-n}^{\operatorname{fr}} \lt V_{r,k} \rt \;\; = \;\; \Omega_{m-n}^{\operatorname{fr}} \, \oplus \, \tilde \Omega_{m-n}^{\operatorname{fr}} \lt V_{r,k} \rt
	\end{equation*}
	where
	\begin{align*}
		m \;\; = \;\; \dim V_{r,k}                   & \;\; = \;\; k(r-k) + k (k-1)/2 & \text{and} \\
		n \;\; = \;\; \dim \overset{(\sim)}{G}_{r,k} & \;\; = \;\; k(r-k).
	\end{align*}
	We may represent this $\omega$-invariant by a generic zero manifold of \ $v \scr \circ \dis f$ \ in the ball \ $B$, \ e.\,g. by \ $\chi \lt \tilde G_{r,k} \rt$ \ many copies of \ $SO(k)$. \
	Hence \ $\omega (f,f)$ \ lies in \ $\Omega_{m-n}^{\operatorname{fr}} \, \oplus \, \{0\} \, \cong \, \pi_{m-n}^S$ \ and is equal to \ $\pm E^\infty \lt i \lt v \scr \circ \dis f, B \rt \rt$ \ (compare \cite{ko7}, 5.7).
	If \ $r \geq 2k \geq 4$ \ we are in the stable dimension range and \ $N (f,f)$ \ or, equivalently, \ $\omega (f,f)$ \ vanishes precisely if \ $i (v \scr \circ \dis f, B)$ \ does (and hence \ $(f,f)$ \ is loose by small deformation).
	On the other hand, if \ $k=1$ \ then \ $f$ \ is the identity map on \ $S^{r-1}$ \ and \ $i(v \scr \circ \dis f,B)$ \ as well as $N (f,f)$ \ vanish precisely when \ $r$ \ is even.

	Similar arguments apply to the canonical projection \ $f$ \ into the Grassmannian \ $G_{r,k}$ \ of \textit{unoriented} $k$-planes in \ $\R^r$.
	\hfill $\Box$
\medskip

\noindent
\textbf{Proof of theorem \ref{thm:9}.} Assume \ $M = S^m$ \ and \ $[f] \in \pi_m (N)$, $m,n \geq 2$.

Then \ $\omega^\# (f,f)$ \ is just as strong as
\begin{equation*}
	\pr_{\dis *} (\omega^\# (f,f)) \; = \; \pm E \scr \circ \dis \partial_N ([f]) \; \in \; \pi_m (S^n)
\end{equation*}
(cf. \cite{ko7}, 5.6 and 5.7). In other words, conditions (iv) and (v) in \ref{thm:9} are equivalent.

If \ $N = S^n/G,$ \ $G \not\cong \Z_2$, \ then \ $N = S^n$ \ (and hence \ $MCC (f,f) = N^\# (f,f)$, cf. \ref{thm:7}c) or \ $n$ \ is odd (and hence \ $N$ \ allows a nowhere vanishing vectorfield and \ $MCC (f,f) = N^\# (f,f) = 0$); indeed, \ $(\# G) \cdot \chi (N) = \chi (S^n) = 2$ \ and \ $\# G \leq 2$ \ for even \ $n$. In both cases conditions (iii) and (iv) in \ref{thm:9} are equivalent.

In general, when a manifold \ $N$ \ is noncompact or has zero Euler characteristic, then it carries also a nowhere vanishing vectorfield and \ $\partial_N (\pi_m(N)) = 0$. If \ $m < 2n-2$, or if \ $m \leq n+4 $ \ and \ $n \equiv 0(2)$, \ then \ $\ker E = 0$ \ except when \ $m = n +4 = 10$. This is seen with the help of the EHP-sequence and Toda's tables (cf. \cite{wh}, XII, 2.3, and \cite{to}). In both cases the Wecken condition \ref{def:3} holds for \ $(m,N)$, and assumptions (i), \dots, (v) in \ref{thm:9} are all equivalent.

In view of \ref{thm:2} this completes the proof of theorem \ref{thm:9}. \hfill $\Box$
\bigskip

\noindent
\textbf{Proof of theorem \ref{thm:2}.} Consider the composed diffeomorphism
\begin{equation*}
	TN \; \xleftarrow[p_{1 \dis *}]{\cong} \; \nu (\Delta , N \times N) \; \cong \; U \; \subset N \times N
\end{equation*}
where \ $\nu (\Delta, N \times N)$ \ and \ $U$, resp., are the normal bundle, and a tubular neighbourhood, resp., of the diagonal \ $\Delta$ \ in \ $N \times N$, and \ $p_{1 \dis *}$ \ denotes the vector bundle isomorphism determined by the first projection \ $p_1 \, \colon \, N \times N \to N$ \ (compare Figure 2.1). This restricts to yield a fiber map from the tangent sphere bundle \ $STN$ \ (cf. \ref{equ:13}) to the configuration space
\begin{equation*}
	\tilde C_2 (N) = \{(y_1, y_2) \in N \times N \; | \; y_1 \not= y_2 \} \; = \; N \times N - \Delta
\end{equation*}
(fibered over \ $N$ \ by \ $p_1$), and relates the two corresponding exact homotopy sequences (cf. the commuting diagram 5.4 in \cite{ko7}). We conclude that the boundary homomorphism in the homotopy sequence of ($\tilde C_2(N), p_1|$) equals \ $j_{N \dis *} \scr \circ \dis \partial_N$. Therefore \ $f$ \ lifts to yield a map \ $(f,f') \, \colon \, S^m \to \tilde C_2 (N)$ (i.\,e. \ $f(x) \not= f'(x)$ \ for all \ $x \in S^m$) if and only if \ $j_{N \dis *} (\partial_N (f)) = 0$. This is equivalent to conditions (i) and (ii) (cf. \ref{thm:9}) when \ $j_{N \dis *}$ \ is injective.

Next assume that \ $\pi_1(N) \not= 0$ \ and \ $m \geq 2$. \ Then the fiber \ $p^{-1}(\{x_0\})$ \ in the universal covering space \ $p \, \colon \, \tilde N \to N$ \ has at least two points \ $y_0, y_0'$. \ Moreover \ $j_N$ \ lifts to yield a composite
\begin{equation*}
	S^{m-1} \xrightarrow{\tilde j_N} \tilde N - p^{-1}(\{x_0\}) \xrightarrow{\incl} \tilde N - \{y_0,y_0'\}
\end{equation*}
which maps \ $S^{m-1}$ \ homeomorphically onto the boundary sphere \ $\partial B'$ \ of a small ball \ $B'$ \ around \ $y_0'$ \ in \ $\tilde N - \{y_0\}$. \ Using an isotopy we can find a ball \ $B$ \ in \ $\tilde N$ \ around \ $y_0$ \ containing \ $B'$ \ and whose  boundary \ $\partial B$ \ meets \ $\partial B'$ \ in precisely one point. Then there exists a (deformation) retraction
\begin{equation*}
	r \; \colon \; \tilde N - \{ y_0, y_0' \} \; \to \; (\tilde N - \stackrel{\circ}{B}) \vee \partial B'
\end{equation*}
and an obvious quotient map \ $q$ \ to \ $\partial B'$ \ such that \ $q \, \scr \circ \dis \, r \, \scr \circ \dis \, \incl \, \scr \circ \dis \, \tilde j_N$ \ is a homeomorphism. Thus the induced homomorphisms \ $\tilde j_{N \dis *}$ \ and \ $j_{N \dis *}$ \ are injective.

Finally consider the case \ $N \, = \, \K P(n')$. \ Here \ $j_N$, \ when composed with the natural retraction from \ $\K P(n') - \{y_0\}$ \ to \ $\K P(n'-1)$, \ is homotopic to the standard projection \ $p \, \colon \, S^{dn' -1} \to \K P(n' -1)$ \ (compare 4.5). If \ $n' -1 \geq 1$, \ then \ $p_{\dis *}$ \ -- and hence \ $j_{N \dis *}$ \ -- is injective (compare theorem \ref{thm:3}). \hfill $\Box$

%% file: Kapitel6.tex
\section{Kervaire invariants and Hopf invariants} \label{sec:6}
\bigskip

	In this section we prove theorems \ref{thm:13} and \ref{thm:14} as well as their corollaries.

	Given a covering space \ $p \, \colon \, \tilde N \to N$, \ the tangent bundle \ $TN$ \ of \ $N$ \ pulls back to \ $p^{\dis *} (TN) \, \cong \, T \tilde N$. \ Therefore
	\begin{equation*}
		\partial_{\tilde N} \; = \; \partial_N \,\scr \circ \dis p_{\dis *} \; \colon \; \pi_m (\tilde N) \; \longrightarrow \; \pi_{m-1} (S^{n-1}), \qquad m \geq 1,
	\end{equation*}
	(compare \ref{equ:13}).

	In particular, if we want to check the Wecken property \ref{def:3} for a spherical space form \ $N = S^n/G$, \ we need to discuss diagram \ref{equ:13} only for the case \ $N \, = \, S^n$. \ Then \ $\partial_N$ \ is the boundary homomorphism in the exact homotopy sequence of the canonical fiber map from the Stiefel manifold \ $V_{n+1,2} \, = \, STS^n$ \ to \ $S^n$. \ For simplicity we write
	\begin{equation*}
		\partial \quad \coloneqq \quad \partial_{S^n} \quad = \quad \partial_{\R}
	\end{equation*}
	(compare \ref{equ:16}).

	When can the pair \ $(S^m, S^n/G)$ \ possibly fail to have the Wecken property \ref{def:3},
	i.\,e. when is the intersection \ $\partial(\pi_m(S^n)) \cap \ker E$ \ nontrivial? Clearly we have to look at nonstable dimension combinations
	where \ $m \geq 2n+2$ (since otherwise \ $\ker E = 0$). If \ $m \leq 3n-5$ \ we can approach our question via the exact EHP-sequence
	\addtocounter{thm}{1}
	\begin{equation}\label{equ:9}
		\xymatrix @C=22pt {
			\cdot \cdot \ar[r] &
			\pi_{m-1} (S^{2n-3}) \ar[r]^{w_{n-1 \dis *}} &
			\pi_{m-1} (S^{n-1})  \ar[r]^{E}       &
			\pi_m (S^n)          \ar[r]^{H} \ar@{-->}@/^1pc/[l]^{\partial} &
			\pi_m (S^{2n-1})     \ar[r]                                &
			\cdots
		}
	\end{equation}
	(cf. \cite{we}, XII, 2.3-2.5), where the homomorphism \ $w_{n-1 \dis *}$ \ is induced by the Whitehead square \ $w_{n-1} \coloneqq [\iota_{n-1}, \iota_{n-1}]$ \ (of the identity map on \ $S^{n-1}$).
	Its kernel as well as the homomorphism \ $H$ \ are described by the (second) Hopf-James invariant.

	Let us study the first nonstable dimension combination \ $m = 2n -2$.
	Here \ $E$ \ is onto, but rarely injective.
	If fact, if \ $n\not\equiv 0(2)$ the kernel of \ $E$ \ is even infinite (since the $\Z$-valued Hopf invariant vanishes on \ $\pi_{m+1} (S^n))$; but this has no effect on Wecken properties since \ $\partial (\pi_m(S^n))=0$ \ whenever $n$ is odd.
	In view of the famous Hopf invariant one theorem of Adams \ $\ker E = 0$ \ precisely for \ $n=2,4$ or 8.

	Thus we need to consider only the remaining case where \ $n\equiv 0(2)$, \ $n \not= 2,4,8$.
	We obtain the short exact sequence
	\stepcounter{thm}
	\begin{equation}
		\xymatrix{
			  0 \ar[r]
			& \Z_2 \ar[r]^-{\cdot w_{n-1}}
			& \pi_{2n-3} (S^{n-1}) \ar[r]^{E}
			& \pi_{2n-2} (S^n) \ar[r] \ar@{-->}@/^1pc/[l]^{\partial}
			& 0 \ ;
		}
	\end{equation}
	\stepcounter{thm}%
	moreover, \ $\partial = 2 E^{-1}$ \ is obtained by taking inverse images and
	multiplying with 2. Thus the Wecken condition
	\begin{equation*}
		\partial (\pi_{2n-2} (S^n)) \; \cap \; \ker E \; \; = \; \; 2 \cdot \pi_{2n-3} (S^{n-1}) \; \cap \; \Z_2 \cdot w_{n-1} \;\; = \;\; 0
	\end{equation*}
	fails precisely if \ $w_{n-1} = [\iota_{n-1}, \iota_{n-1}]$ \ ``can
	be halved'', i.\,e. lies in \ $2\cdot \pi_{2n-3} (S^{n-1})$. But according to
	\cite{gr1} this is equivalent to the existence of an element in the group \
	$\pi_{2n-2} (S^n) \cong \pi_{n-2}^S$ \ which has order 2 and Kervaire invariant one.
	In view of further details in section 3 of \cite{gr1} and in our example \ref{exa:5}, and since \ $2 \ker(E \circ \delta) =0$, \ theorem \ref{thm:13} and corollary \ref{cor:3} follow.
	\hfill $\Box$
	\medskip

	In the next nonstable dimension setting assume that \ $n \equiv 2(4), n \geq 6$ \ (otherwise the Wecken condition 1.18 holds, cf. \cite{ja2}, 3.5). Then the EHP-sequence \ref{equ:9} yields the shorter exact sequence
		\begin{equation}
			\xymatrix @R=-.1pc {
				  0 \ar[r]
				& \Z_2 v_{n-1} \quad \subset \quad \pi_{2n-2} (S^{n-1}) \ar[r]^-{E}
				& \pi_{2n-1} (S^n) \ar[r]^-{\frac12 H} \ar@{-->}@/^1pc/[l]^{\partial}
				& \Z \ar[r]
				& 0   \\
				&& \rotatebox{90}{$\in$} \\
				&& w_n
			}
		\end{equation} 
		where $v_{n-1} \coloneqq [\iota_{n-1}, \eta_{n-1}]$ \ , \ $w_n=[\iota_n, \iota_n]$
		\ and \ $\eta_{n-1}$ \ generates \ $\pi_n (S^{n-1})$; moreover
		\begin{equation*}
			\partial w_n = v_{n-1} \not\in 2\pi_{2n-2} (S^{n-1}) = \partial \circ E
			(\pi_{2n-1} (S^{n-1}))
		\end{equation*}
		(cf. \cite{ja2}, lemmas 3.5, 7.4, 5.2 and 3.6). Since the Hopf invariant of \ $w_n$
		\ equals \ $\pm 2$ \ we obtain the splitting
		\begin{equation*}
			\pi_{2n-1} (S^n) \cong E (\pi_{2n-2} (S^{n-1})) \oplus \Z w_n.
		\end{equation*}

		If \ $MCC (f_1, f_2) \not = N^\# (f_1, f_2)$ \ we see (from 1.19) that \ $f_1 \sim f_2$ \ (cf. \ref{thm:11}) can be lifted to yield a class of the form
		\begin{equation*}
			[\tilde f] = E(\alpha) \pm \frac12 H(\tilde f) w_n \; \in \pi_{2n-1} (S^n), \qquad \alpha \in \pi_{2n-2} (S^{n-1}),
		\end{equation*}
		such that
		\begin{equation*}
			\partial ([\tilde f]) = 2\alpha + \frac12 H(\tilde f) v_{n-1} \not= 0
		\end{equation*}
		but \ $E(\partial ([\tilde f])) = E(2 \alpha)$ \ vanishes (and so does \ $2\alpha$ \ since \ $v_{n-1}$ \ cannot be halved).
		Then \ $H(\tilde f) \equiv 2(4)$ \ and \ $MCC (f_1, f_2) = 1$.

		Conversely, if \ $[f_1] = [f_2]$ \ lifts to an odd multiple of \ $[\iota_n, \iota_n]$ \ then \ $1 = MCC (f_1, f_2) \not= N^\# (f_1, f_2) = 0$ \ since \ $\partial([\tilde f]) = v_{n-1} \in \ker E$.
		\hfill $\Box$
		\medskip

	Note that the various versions of Hopf invariants (e.\,g. \`a la James, Hilton, Ganea, \dots) are crucial not only in the results of this section and their proofs (in the form of the EHP-sequence). They seem to make frequent appearances all over topological coincidence theory (see e.\,g. also theorem \ref{thm:7}, the discussion of 1.35, as well as proposition \ref{prop:2}).

%% file: Kapitel7.tex
\section{Roots and degrees}\label{sec:7}

	In this section we discuss our coincidence invariants for pairs of the form \ $(f, *)$ \ where \ $f \, \colon \, M \to N$ \ is a map between smooth connected manifolds as in 1.36 and \ $*$ \ denotes the constant map with value \ $* \in N$. \

	First note that the pathspace \ $E (f, *)$ \ (which plays such a central role in the definition of Nielsen numbers, cf. \ref{equ:12} and \ref{def:2}) is just the mapping fiber of \ $f$ \ (cf. \cite{wh}, I. 7, p. 43).
	The Reidemeister number \ $\# \pi_0 (E (f,*) )$ \ equals the cardinality of the cokernel of the induced homomorphism
	$f_{\dis *} \, \colon \, \pi_1 (M) \to \pi_1 (N)$ \ (cf. \ref{def:1} and \ref{equ:1}).

	Our $\omega$-invariants give rise to four types of degrees (compare diagram \ref{dia:1}).

	\begin{defn}\label{def:5}
		\begin{alignat*}{3}
			\deg^\# (f)          \;& \coloneqq \; \omega^\# (f,*)             & \in \quad  &&& \Omega^\# (f,*) ; \\
			\widetilde{\deg} (f) \;& \coloneqq \; \tilde \omega (f,*)         & \in \quad  &&& \Omega_{m-n} (E (f,*) ; \tilde \varphi ); \\
			\deg (f)            \; & \coloneqq \; \omega (f,*)                & \in \quad &&& \Omega_{m-n} (M ; \varphi); \\
			\deg^\Z (f)         \; & \coloneqq \; g_{\dis *}([C (f, \dis *)]) & \quad  \in \quad &&& H_{m-n} (M ; \widetilde{\Z}_{\varphi}).
		\end{alignat*}
	\end{defn}
	All of our coincidence invariants are compatible with homotopies.
	Thus we may assume that \ $f$ \ is smooth and has \ $* \in N$ \ as a regular value.
	Then the four degrees defined in \ref{def:5} capture appropriate coincidence data of the submanifold \ $C(f, *) \, = \, f^{-1} (\{ * \})$ \ of \ $M$ \ (compare \ref{equ:2} - \ref{equ:14}) with decreasing accuracy.
	For example, if \ $M$ \ and \ $N$ \ are oriented manifolds having the same dimension \ $m = n$ \ then \ $\deg^\# (f)$ \ and \ $\widetilde \deg (f)$ \ still keep track of the Nielsen decomposition of \ $C (f, *)$, \ but
	\begin{alignat*}{4}
		& \deg (f)    & \in \quad &&& \Omega_0 (M) & = \Z & \qquad\text{and} \\
		& \deg^\Z (f) \quad &  \in \quad &&& H_0 (M ; \Z) \;& = \Z
	\end{alignat*}
	yield just another description of the standard mapping degree.

	Given \ $x \in C(f,*)$ \ and a pathcomponent \ $Q$ \ of \ $E (f,*)$, \ there is a loop \ $\theta$ \ in $N$ at $*$ such that \ $(x, \theta) \in Q$. \ Clearly \ $\theta$ \ defines a homotopy of constant maps and hence induces a fiber homotopy equivalence from \ $E (f,*)$ \ to itself (cf. \cite{ko4}, 3.2) which maps the pathcomponent \ $Q_0$ \ (containing ($x$, constant loop)) to \ $Q$. \
	Similarly, $\theta$ induces selfbijections of the sets which are listed in diagram \ref{dia:1} (cf. e.\,g. \cite{ko4}, 3.3).
	Therefore the contributions of \ $Q_0$ \ and \ $Q$ \ to a given $\omega$-invariant and Nielsen number (cf. \ref{def:2}) are either both trivial or both nontrivial.
	We have proved

	\begin{prop}\label{prop:3}
		The Reidemeister number \ $\# \pi_0 (E (f,*)) \, = \, \#\!\left(\pi_1 (N) / f_{\dis *} (\pi_1 (M) ) \right)$ \ and \ $0$ \ are the only values which any of the four Nielsen numbers of \ $(f, *)$ \ can possibly take.
	\end{prop}

	\begin{example}$\mathbf{M = S^m.}$
		\textit{Here we assume that \ $m,n \geq 2$} \ (the remaining cases being well understood \ -- when \ $M = N = S^1$, \ cf. \ref{thm:7} -- \ or trivial).

		Since the minimum and Nielsen numbers are (free) homotopy invariants we may also assume that the maps \ $f_1, f_2, \dots$ \ have a convenient base point behavior.
		Thus fix base points \ $x_0 \in S^m$ \ and \ $y_1 \not= y_2 = * \in N$, \ and choose a local orientation of \ $N$ \ at the point \ $*$, \ as well as a path \ $\tau$ \ in \ $N$ \ joining \ $*$ \ to \ $y_1$. \
		For any two maps \ $f_i \, \colon \, (S^m , x_0) \to (N, y_i), \; i=1,2$, \ these choices (and the Pontryagin-Thom procedure) allow us to identify the set \ $\Omega^\# (f_1,f_2)$ \ with the homotopy group \ $\pi_m (S^n \wedge ( \Omega N)^+)$ \ (cf. section 6 in \cite{ko6} which uses an observation of A. Hatcher and F. Quinn, cf. \cite{hq}, 3.1).
		Here \ $\Omega N$ \ denotes the space of loops in \ $N$ \ starting and ending at \ $*$, \ and \ $S^n \wedge (\Omega N)^+$ is just the Thom space of the trivial $n$-plane bundle over \ $\Omega N$.

		These identifications present enormous advantages.
		First of all we can now compare the \ $\omega^\#$-invariants of different pairs of maps.
		Moreover we are greatly helped by compatibilities with group structures.
	\end{example}

	\begin{lem}\label{lem:1}
		Given maps
		\begin{equation*}
			f_i, f_i' \;\; \colon \;\; (S^m , x_0) \;\; \to \;\; (N, y_i) \quad, \qquad i= 1,2,
		\end{equation*}
		we have
		\begin{equation*}
			\omega^\# (f_1 + f_1', f_2 + f_2') \;\; = \;\; \omega^\# (f_1, f_2) + \omega^\# (f_1', f_2').
		\end{equation*}
		In particular,
		\begin{equation*}
			\deg^\# \;\; = \;\; \omega^\# (-,*) \;\; : \;\; \pi_m (N, y_1) \;\; \to \;\; \pi_m (S^n \wedge (\Omega N)^+),
		\end{equation*}
		is a group homomorphism.
	\end{lem}

	\begin{proof}
		The coincidence set \ $C (f_1 + f_1', f_2 + f_2')$ \ consists of the two parts \ $C(f_1,f_2)$ \ and \ $C(f_1',f_2')$ \ which lie in different halfspheres of \ $S^m$ \ and are separated by an equator.
		The second claim is just the special case \ $f_2 \equiv f_2' \equiv *$.
	\end{proof}

	Our choice of the path \ $\tau$ \ from \ $*$ \ to \ $y_1$ \ determines also an isomorphism \ $\pi_m (N, *) \, \cong \, \pi_m (N, y_1)$ \ which we denote by \ $[f] \, \to \, [f^\tau]$. \
	Moreover there is a canonical involution \ $\operatorname{inv}$ \ of the group \ $\pi_m (S^n \wedge (\Omega N)^+)$ \ such that e.\,g.
	\begin{equation*}
		\omega^\# (y_1,f) \;\; = \; \; \operatorname{inv} (\deg^\# ( [f^\tau] )
	\end{equation*}
	for all \ $[f] \in \pi_m(N, *)$ \ (cf. \cite{ko6}, 2.5 and 6.1).
	In particular, given any pair \ $(f_1, f_2)$ \ as in lemma \ref{lem:1}, we see that
	\begin{align*}
		\omega^\# (f_1,f_2) & = \omega^\# (f_1, *) + \omega^\# (y_1, f_2) \\
		                    & = \deg^\# \lt [f_1] \rt + \operatorname{inv} \scr \circ \dis \deg^\# \lt [f_2^\tau] \rt.
	\end{align*}
	Thus, in a way, the root case invariant \ $\deg^\#$ \ is basic for all other $\omega^\#$-invariants.

	Lemma \ref{lem:1} also allows us to split \ $\omega^\# (f_1,f_2)$ \ into a `root component` and a `selfcoincidence component`:
	\begin{equation*}
		\omega^\# (f_1,f_2) \;\; = \;\; \deg^\# (f_1 - f_2^\tau) + \omega^\# (f_2^\tau, f_2)
	\end{equation*}
	In the many cases when \ $N^\# (f_2,f_2)$ \ or, equivalently, \ $\omega^\# (f_2^\tau, f_2)$, \ vanishes (e.\,g. if at least one of the four restrictions in theorem \ref{thm:12} is not satisfied) only the root component survives; then \ $N^\# (f_1, f_2) \, = \, N^\# (f_1 - f_2^\tau, *)$ \ and the other three Nielsen numbers are equal to \ $\# \pi_1 (N)$ \ or \ $0$ \ (cf. \ref{prop:3}).
	On the other hand, if \ $N^\# (f_2, f_2) \, \not= \, 0$, \ then the Reidemeister number \ $\# \pi_1 (N)$ \ (and hence all four Nielsen numbers) of \ $(f_1, f_2)$ \ are bounded above by $2$ (cf. \ref{thm:12}). In particular, this proves theorem \ref{cor:5}.
	\hfill $\Box$
	\bigskip

\subsection*{Acknowledgment}
		It is a pleasure to thank D. Randall for many very stimulating discussions.